\documentclass[11pt]{article}
\usepackage[colorlinks=false]{hyperref}
\usepackage{mystuff}

\usepackage[font=sf]{caption}

\newcommand\tstep{{{\Delta}t}}

\title{Gauss-quadrature method for one-dimensional mean-field SDEs}
\author{Peter Kloeden\footnote{School of Mathematics \& Statistics,
		Huazhong University of Science \& Technology,
		Wuhan 430074, China.} \and Tony Shardlow\footnote{Department of Mathematical Sciences, University of Bath, Bath BA2 7AY, UK. \texttt{t.shardlow@bath.ac.uk}  TS is grateful for support from Huazhong University of Science \& Technology, where part of this work was performed.  }}

\numberwithin{theorem}{section}

\begin{document}
\maketitle
\begin{abstract} Mean-field SDEs, also known as McKean--Vlasov equations, are stochastic differential equations where the drift and diffusion depend on the current distribution in addition to the current position. We describe an efficient numerical method for approximating the distribution at time $t$ of the solution to the initial-value problem for one-dimensional mean-field SDEs. The idea is to time march (e.g., using the Euler--Maruyama time-stepping method) an $m$-point Gauss quadrature rule. With suitable regularity conditions, convergence with first order is proved for Euler--Maruyama time stepping. We also estimate the work needed to achieve a given accuracy in terms of the smoothness of the underlying problem. Numerical experiments are given, which show the effectiveness of this method as well as two second-order time-stepping methods. The methods are also effective for ordinary SDEs in one dimension, as we demonstrate by comparison with the multilevel Monte Carlo method.
\end{abstract}

\section{Introduction}


McKean--Vlasov or mean-field SDEs are a class of stochastic differential equations where the drift and diffusion depend on the current position along the path and on the current distribution. They were derived to describe propagation of chaos in a system of particles that interact only by their empirical mean in the limit of large number of particles \cite{McKean1966-kb}. We study mean-field SDEs in one dimension and are interested in the following initial-value problem: determine the real-valued process $X^\mu(t)$, $t>0$, such that
\begin{equation}
X^\mu(t)-X^\mu(0)%
=\int_0^t \int_{\real}  a(X^\mu(s), y)\, P_s^\mu(dy)\,ds%
+  \int_0^t\int_{\real} b(X^\mu(s),
y)\,P_s^\mu(dy)\,dW(s),%
\label{eq:mf_sde}
\end{equation}
where $P_s^\mu$
denotes the distribution of $X^\mu(s)$ and the initial distribution $X^\mu(0)\sim \mu$ for some prescribed probability measure $\mu$. Here, $a\colon\real^2\to\real$ is the drift, $b\colon\real^2\to\real$ is the diffusion,  $W(t)$ is a one-dimensional Brownian motion (independent of $X^\mu(0)$), and we interpret the stochastic integral as an Ito integral. We also write this as
\[
dX^\mu(t)=P^\mu_t (a(X^\mu(t),\cdot))\,dt + P^\mu_t (b(X^\mu(t),\cdot))\,dW(t),\qquad X^\mu(0)\sim \mu,
\]
where $\nu(\phi)\coloneq\int_\real \phi(x)\,\nu(dx)$ for an integrable function $\phi\colon \real\to\real$ and a measure $\nu$ on $\real$. Under the following condition, \cref{eq:mf_sde} has a unique strong solution with a smooth density \cite[Theorem 2.1]{Antonelli2002-er}.
(Though \cref{eq:mf_sde} is well-posed more generally \cite{Gartner1988-xc,sznit,peter2}, \cref{ass1} is close to the ones in our error analysis.)
\begin{assumption}\label[assumption]{ass1}
	Suppose that $p$th moments of  the initial distribution $\mu$ are finite for all $p\ge 1$ and that the coefficients $a$ and $b$ are smooth with all derivatives uniformly bounded.
\end{assumption}

Several numerical methods have been proposed for \cref{eq:mf_sde} and their convergence behaviour analysed. Early work includes \cite{Bossy1997-fs,Bossy1996-gr}, which show convergence of a method based on Monte Carlo evaluation of the averages and Euler--Maruyama time stepping. The same method was studied using Malliavin calculus in \cite{Antonelli2002-er} and more refined convergence results proved. More recently, \cite{Ricketson2015-xv} has developed the multilevel Monte Carlo method in cases where the drift and diffusion depend on the distribution via the mean of a function of $X^\mu(t)$. Cubature methods have also been developed in \cite{McMurray2015-lv}.

We are interested in numerical approximation of  the distribution of $X^\mu(t_n)$ by a probability measure $Q_n$, where $t_n=n\tstep$ for a time step $\tstep>0$.  Consider a one-step numerical method  that pushes forward the measure $Q_n$ to $Q_{n+1}$. For an example, let
\begin{equation}\label{psi}
\Psi(x,\tstep,Q)%
\coloneq x%
+ \tstep\,Q (a(x,\cdot))%
+\sqrt{\tstep}\, Q( b(x,\cdot))\,\xi,%
\end{equation}
for $\xi\sim \Nrm(0,1)$ or a random variable with a nearby distribution, such as the two-point random variable with $\prob{\xi=\pm 1}=1/2$.  For the Euler--Maruyama method,
 $Q_{n+1}$ is the distribution of  $X_{n+1}=\Psi(X_n,\tstep, Q_n)$, assuming $\xi$ is independent of $X_n$ and $X_0\sim \mu$.
In the case that $a,b$ are independent of their second argument,
\[
X_{n+1}=X_n + \tstep \,a(X_n)+ \sqrt{\tstep} \,b(X_n) \,\xi_n,
\]
where $\xi_n$ are \iid copies of $\xi$,
which is the standard Euler--Maruyama method. For ordinary SDEs, it is well-known that first-order weak convergence results if $a,b$ and the test function $\phi\colon\real\to\real$ are sufficiently smooth  \cite{Kloeden2011-qd}:
\[
\mean{\phi(X^\mu(1))}- \mean{\phi(X_N)}%
=P^\mu_1(\phi) -Q_N(\phi)%
=\order{\tstep},%
\qquad t_N=1.
\]

This method is of limited practical value for approximating $P^\mu_t(\phi)$.
 The support of $Q_N$ is uncountable if Gaussian random variables $\xi$ are used or otherwise countable but very large in number, and  the expectation $Q_N(\phi)$ is usually approximated via a Monte Carlo
method that samples from $Q_N$.  For the mean-field SDE, this is  more problematic, as all the particles must be tracked at the same time as $Q_n(a(X_n,\cdot))$ and $Q_n(b(X_n,\cdot))$ must be evaluated at each time step.

In this paper, we explore an alternative to Monte Carlo integration and employ instead Gauss quadrature, which provides accurate quadrature rules
that converge rapidly in the number of quadrature points,
under smoothness criterion on the integrand.
The idea then is to replace $Q_{n}$ by an $m_n$-point Gauss
quadrature  and thereby
reduce the number of points that we follow with the time stepping.
That is, we propagate weights $w^i_n$ and quadrature points $x^i_n$ of an
$m_n$-point rule $Q_n$, and approximate
\[
P^\mu_1(\phi)%
\approx Q_N(\phi)%
\coloneq\sum_{i=1}^{m_N} w^i_N \,\phi(x_N^i),%
\qquad t_N=1.
\]
We
derive a choice of $m_{n}$ in \cref{err} that gives first-order convergence for smooth problems.  The computation of the
Gauss quadrature rules is very efficient using standard
algorithms \cite{MR0245201,MR2061539,Boley1987-jt}. This leads to  numerical  methods for mean-field SDEs that are very efficient and we find methods that require $\order{\abs{\log \epsilon}^{3}/\epsilon}$ work to achieve accuracy $\epsilon$ for mean-field SDEs with smooth coefficients and initial distributions (see \cref{blimey2}).  This compares favourably with the $\order{1/\epsilon^2}$ work required for multilevel Monte Carlo methods, as we see in \cref{example:gbm}.

Mean-field SDEs arise as reduced-order models for systems of interacting particles. The drift and diffusion are defined in terms of the distribution of $X(t)$, so that moments of $X(t)$ can be included in their definition. In other words, the interaction with the ensemble of particles is approximated by moments and mean-field SDEs, including one-dimensional mean-field SDEs, are of interest in studying high-dimensional systems. The techniques in this paper apply to mean-field SDEs in one spatial dimension, as  Gauss quadrature is most natural for integrals over the real line, where algorithms are readily available to compute the quadrature rule. In principle, the methods and theory extend to higher dimensions, though it would be difficult to compute a suitable cubature rule. It would require a cubature rule that can be easily computed and satisfies Gauss quadrature-type error estimates (see \cref{classic_gq}). These are currently unavailable (see \cite{Xu2015-vr} for a recent discussion of Gaussian cubature).

This paper is organised as follows: \cref{sec:g} reviews key facts about Gauss quadrature and develops preliminary lemmas. \cref{sec:alg} describes the method for Gauss quadrature with Euler--Maruyama time stepping, which we call the GQ1 method. The error analysis for stochastic ODEs is developed in  \cref{err}, where we show how to choose the number $m_n$ of Gauss points. In \cref{err2}, we extend the error analysis to mean-field SDEs and modify the choice of $m_n$ for this case. We also discuss a straight-forward generalisation of the methodology to the initial-value problem for
\begin{gather}
  \begin{split}
X^\mu(t)-X^\mu(0)%
&=\int_0^t A\pp{\int_{\real}  a(X^\mu(s), y)\, P_s^\mu(dy)}\,ds%
\\&\qquad +  \int_0^t B\pp{\int_{\real} b(X^\mu(s),
y)\,P_s^\mu(dy) }\,dW(s),%
\end{split}\label{eq:mf_sde_ref}
\end{gather}
for smooth functions $A,B\colon \real\to\real$, which allows a nonlinear dependence on the time-$t$ distribution. In \cref{num}, we describe two extensions of GQ1: namely, GQ1e, which uses GQ1 with extrapolation, and GQ2, which use Gauss quadrature with a second-order time-stepping method. The remainder of the section gives a number of numerical experiments, including a comparison with the multilevel Monte Carlo method for ordinary SDEs.
\subsection{Notation}

For a measure $\mu$ on $\real$ and an  integrable function
$\phi\colon\real\to\real$, denote
$\mu(\phi)\coloneq \int_{\real} \phi(x)\,\mu(dx)$.  Let $C^k(\real^d)$ denote the space of $k$-times continuously differentiable real-valued functions on $\real^d$ and
 $F^{k,\beta}\coloneq \{\phi\in C^k(\real^d)\colon \norm{\phi}_{k,\beta}<\infty\}$, where
 \[
\norm{\phi}_{k,\beta}%
\coloneq \max_{0\le \abs{\alpha}\le k} %
\sup_{x\in\real^d}
\frac{\abs{\phi^{(\alpha)}(x)}}%
{1+\abs{x}^\beta},
\]
using the multi-index notation.
Let $C^k_K(\real^d) %
\coloneq\{\phi\in C^k(\real^d)\colon%
	\norm{\phi^{(\alpha)} }_\infty%
	\le K, \quad %
	0\le \abs{\alpha}\le k\}$,
where $\norm{\cdot}_\infty$ denotes the supremum norm. Throughout the paper, we use $c$ as a generic constant that varies from place to place.

\section{Gauss quadrature and error estimates}\label[section]{sec:g}
Before describing the algorithm, we review Gauss quadrature and associated error estimates. Let ${\cal P}_n$ denote the
polynomials up to degree $n$.

\begin{definition}[Gauss quadrature]
	We say weights $w^i>0$ and points $x^i\in\real$ for
	$i=1,\dots,m$ define an $m$-point Gauss quadrature rule
	with respect to a measure $\mu$ on $\real$ if
	\[
	\int_\real p(x)\,\mu(dx)%
	=\sum_{i=1}^m w^i\, p(x^i),\qquad %
	\forall \,p\in {\cal P}_{2m-1}.
	\]
\end{definition}

The $m$-point Gauss quadrature rule for a discrete measure
\[
\mu%
=\sum_{i=1}^N %
v^i\, \delta_{y^i},
\]
with weights $v^i>0$ and points $y^i$,
can be found via the three-term recurrence
relation for the orthogonal polynomials corresponding to the inner product $\ip{f,g}_\mu\coloneq\int_\real f(x) \,g(x)\, \mu(dx)$. First,  form the matrix $A$
with diagonal $[1,y^1,\dots,y^N]$ and first row and column
given by $[1,\sqrt{v^1},\dots,\sqrt{v^N}]$  (all other entries
zero). By applying orthogonal transformations, reduce $A$ to a symmetric tridiagonal matrix with diagonal $[\alpha^0,\alpha^1,\dots,\alpha^N]$ and off-diagonal $[{\beta}^0,{\beta}^1,\dots,\beta^N]$. The $\alpha^i$ and $\beta^i$ define the
three-term recurrence relation. Next define the Jacobi matrix, which is the symmetric tridiagonal matrix with diagonal
$[\alpha^0,\alpha^1,\dots]$ and off-diagonals $[\sqrt{\beta^0},\sqrt{\beta^1},\dots]$.  To find the $m$-point Gauss quadrature rule, the leading $m\times m$ submatrix of the Jacobi matrix should be chosen. Its eigenvalues determine the
quadrature points and the first component of the normalised
eigenvectors determine the weights, as given by the
well-known Golub--Welsch algorithm. See \cite{Boley1987-jt,MR0245201,MR2061539}.

Thus, to compute the $m$-point Gauss quadrature rule for an $N$-point discrete measure, we reduce the original matrix $(N+1)\times(N+1)$ matrix $A$
to tridiagonal form using a Lanczos procedure and solve a symmetric
eigenvalue problem for an $m\times m$ matrix. The complexity is $\order{N^2+m^3}$, which becomes burdensome when either $m$ or $N$ are large. It is the rapid convergence properties of Gauss quadrature that enable us to control the problem size.

Let us describe the errors for Gauss quadrature.
For an integrable function $\phi\colon\real\to\real$, denote the
approximation error
\[
E(\phi)=
\int_\real \phi(x)\,\mu(dx)%
-\sum_{i=1}^m w^i \,\phi(x^i).%
\]

\begin{theorem}\label{classic_gq}
  Let $\phi\in C^{2m}(\real)$. The error for $m$-point Gauss quadrature is
	\[
	E(\phi)%
	= \frac{\phi^{(2m)}(\xi)}{(2m)!} \,\ip{p_m,p_m}_\mu,\qquad \forall\,\phi\in C^{2m}(\real),
	\]
	for some $\xi\in\real$, where $\ip{p_m,p_m}_\mu=\int_\real
	p_m(x)^2\,\mu(dx)$, $p_m(x)=(x-x^1)\cdots(x-x^m)$, and $x^i$ are the Gauss quadrature points.
\end{theorem}
\begin{proof}
	See \cite[Theorem 3.6.24]{Stoer2010-ku}.
\end{proof}

This theorem shows that Gauss quadrature converges rapidly as the number of points $m\to\infty$ for smooth integrands $\phi$. We require the following alternative characterisation of the error in terms of a minimax polynomial. A similar result is available for continuous measures in \cite[Theorem 5.4]{Atkinson1989-zq}.
\begin{theorem}\label{lgqe} Consider a discrete probability measure $\mu=\sum_{i=1}^N v^i \,\delta_{y^i}$ and approximation by the $m$-point Gauss quadrature rule $\sum_{i=1}^m w^i \,\delta_{x^i}$.
The absolute error
	\[
	\abs{E(\phi)}%
	\le
	 \min_{p\in \mathcal{P}_{2m-1}} %
	 \bp{\max_{i=1,\dots,N} \abs{p(y^i)-\phi(y^i)}%
		+ \max_{i=1,\dots,m}\abs{p(x^i)-\phi(x^i)}}.
	\]

\end{theorem}
\begin{proof}
	Let $p\in \mathcal{P}_{2m-1}$. As $m$-point Gauss quadrature is exact for $p\in \mathcal{P}_{2m-1}$,
	\begin{align*}
	E(\phi)&
	=E(\phi-p)%
	=\sum_{i=1}^N v^i\,(\phi-p)(y^i)%
	-\sum_{i=1}^m w^i \,(\phi-p)(x^i),
	\end{align*}
	so that
	\[
	\abs{E(\phi)}
	\le \sum_{i=1}^N v^i \max_{i=1,\dots,N}\abs{\phi(y^i)-p(y^i)}%
	+\sum_{i=1}^m w^i \max_{i=1,\dots,m}\abs{\phi(x^i)-p(x^i)}.
	\]
	Since $\sum_{j=1}^N v^i=\sum_{j=1}^m w^i=1$, this completes the proof.
\end{proof}

For the numerical solution of SDEs, we are interested in the discrete measure generated by applying Euler--Maruyama with a two-point approximation to the Gaussian increment, which increases the number of points in the support by a factor of two on each step. Using the resulting tree structure, the support can be grouped into points that stem from a smaller set of points. We write down a special error estimate in this setting.

\begin{corollary}\label[corollary]{c1}  Let $\mu$ be a discrete measure with support $\{y^1,\dots,y^{N m}\}$ and consider approximation by $m$-point Gauss quadrature. Suppose that there exists $z^i$ such that \[
\max_{j=(i-1)N+1,\dots,iN}	\abs{ z^i-y^j }%
	\le \delta, \quad\text{for $i=1,\dots,m$.}
	\]
	 Then,
	\[
	\abs{ E(\phi) } %
	\le \delta\, (2\,R)^{2m -1}\,%
	\frac{1}{(2 m)!}\,%
	 \sup_{x\in(-R,R)}\abs{\phi^{(2 m)} (x)},\qquad %
	 \forall \phi\in C^{2 m}(\real),
	\]
	where $R=\max\{\abs{z^i}, \abs{y^j}\colon i=1,\dots,m,\;j=1,\dots,Nm\}$.
	\end{corollary}
	\begin{proof}
Consider interpolation of $\phi$ by $p\in \mathcal{P}_{2m-1}$ based on the $2m$ interpolation points $z^1,\dots,z^{m},x^1,\dots,x^m$, where $x^i$ denote the Gauss quadrature points.  The error at $y^j$ satisfies
	\[
	p(y^j)-\phi(y^j)%
	 = \bp{ (y^j-z^1)\cdots(y^j-z^m ) \,(y^j-x^1)\cdots(y^j-x^{m}) }\frac{1}{(2 m)!} \,\phi^{(2 m)}(\xi),
	\]
	for some $\xi\in(-R,R)$ (by standard error analysis for Lagrange interpolation).
	In the  product, for each $j$, one term is bounded by $\delta$. Each $\abs{y^j-z^i}\le 2\,R$ by definition of $R$.  Hence,
	\[
	\max_j
	\abs{p(y^j)-\phi(y^j)}%
	\le \delta\,%
	 (2 R)^{2m-1} \,%
	 \frac{1}{(2 m)!} \,%
	 \sup_{x\in(-R,R)}\abs{\phi^{(2 m)} (x)}.
	\]
	The polynomial $p$ is exact at $x^i$	and \cref{lgqe} completes the proof.
	\end{proof}
\section{GQ1: Gauss quadrature with Euler--Maruyama}\label[section]{sec:alg}

We explain now in detail our  method: initialise $Q_0$ with a discrete approximation,
\[
Q_0=\sum_{i=1}^{m_0} w^i_0\, \delta_{x^i_0},
\]
 to the initial distribution $\mu$.  In the case that $\mu=\delta_x$ or $X^\mu(0)=x$ for some known $x\in\real$,  take the
one-point quadrature rule with  $x_0^1=x$, weight
$w^1_0=1$, and $m_0=1$.

Suppose that the weights $w^i_n$ and points $x^i_n$ of $Q_n$ are
known at step $n$. To determine $Q_{n+1}$, generate
the Euler--Maruyama points $X_{n+1}^{i\pm}$ defined by
\begin{equation}\label{wales}
  X^{i\pm}_{n+1}%
  =x_n^i%
  +Q_n (a(x_n^i,\cdot))\,\tstep%
  \pm Q_n (b(x_n^i,\cdot) ) \,\tstep^{1/2},\qquad %
  i=1,\dots,m_n,
\end{equation}
and define the corresponding weights
$W^{i\pm}_{n+1}=w^i_n/2$. Together the points
$X_{n+1}^{i\pm}$ and weights $W_{n+1}^{i\pm}$ define a $2\,m_n$-point quadrature rule, which we denote $Q_{n+1}^\pm$.
If left unchecked, this leads to a $2^n$-factor increase in the size of the quadrature rule, which becomes costly.

 At each step, we may continue with $Q_{n+1}=Q_{n+1}^\pm$ (if the number of points is acceptable or the final time is reached) or approximate and reduce the number of points using Gauss quadrature. To approximate, we do the following:
\begin{algorithm}
\begin{enumerate}
	\item Choose a support $[-R, R]$.
	\item For  $\abs{X_{n+1}}\ge R$, generate two points at $\pm R$ with weights $\sum_{\pm X_{n+1}^{j}\ge R} W_{n+1}^{j}$.
	\item For $\abs{X_{n+1}} < R$, generate the $m_{n+1}$-point Gauss quadrature rule for the measure $Q_{n+1}^\pm$ restricted to $(-R,R)$ (i.e., for the measure $Q_R(\cdot)=Q_{n+1}^\pm(\cdot \cap (-R, R))$).
	\item Combine the points and weights, to define a $(m_{n+1}+2)$-point quadrature rule $Q_{n+1}$.
\end{enumerate}
\label[algorithm]{alg}
\end{algorithm}

The iteration is repeated until the final time is reached. 

Following an error analysis in the next sections, we give formulae  for the number of points $m_n$ and support radius $R$ in terms of $\tstep$ and $t_n$.
First, we establish conditions for boundedness of moments for $Q_n$.

\begin{lemma}\label[lemma]{mf_mom}
	Suppose that $a,b\in C^0_K(\real^2)$ and that $Q_0(e^{\alpha \,x^2})<\infty$ for some $\alpha>0$. Then, for some $c,\lambda>0$ independent of $\tstep$,
	\[
	Q_n(e^{\lambda \,x^2})	%
	\le c\qquad \forall\,t_n\le 1.
	\]
\end{lemma}
\begin{proof}  Consider $\Psi$ defined in \cref{psi} where $\xi$ is the two-point random variable given by $\prob{\xi=\pm 1}=1/2$.
Let $X_{n+1}=\Psi(X,\tstep,Q_n)$ for a fixed value $X$. Then
\begin{align*}
X_{n+1}^2%
&\le X^2%
 + 2\,\tstep\, Q_n(a(X,\cdot)) \,X%
  + \tstep^2\, Q_n(a(X,\cdot))^2%
   + \tstep\, Q_n(b(X,\cdot))^2\\%
&\qquad    +2 (X+\tstep\, Q_n(a(X),\cdot))\, Q_n(b(X,\cdot)) \sqrt{\tstep}\,\xi\\
   &\le X^2%
   + \tstep\, Q_n(a(X,\cdot))\, (X^2+1)%
   + \tstep^2\, Q_n(a(X,\cdot))^2%
   +  \tstep\, Q_n(b(X,\cdot))^2 \\%
   &\qquad+ 2\,(X+\tstep\, Q_n(a(X,\cdot))) \,Q_n(b(X,\cdot))\, \sqrt{\tstep}\,\xi.
\end{align*}
Hence, as $a,b$ are bounded by $K$,
\begin{align*}
X_{n+1}^2    &\le X^2 \,(1 +\tstep\, K)%
    + \tstep\, K%
    + \tstep^2 \,K^2%
    + \tstep\, K^2%
    \\&\quad+ 2\, (X+\tstep\, Q_n(a(X,\cdot)))\, Q_n(b(X,\cdot)) \sqrt{\tstep}\,\xi\\
    &\le X^2\,(1+K\,\tstep) + c\,\alpha\,\tstep+
     2\,(X+\tstep\, Q_n(a(X,\cdot))) \,Q_n(b(X,\cdot))\, \sqrt{\tstep}\,\xi.
\end{align*}
Note that $(e^{x}+e^{-x})/2\le e^{x^2}$ for $x\in\real$ and
\begin{align*}
\mean{e^{\alpha \,X_{n+1}^2} }%
	&\le e^{\alpha\,   \,X^2\,(1+\tstep \,K) + c\, \alpha\, \tstep  }
	\,{ e^{4\,\alpha^2 \,(X+\tstep \, Q_n(a(X,\cdot) ))^2\, \tstep\,  Q_n(b(X,\cdot))^2}}.
	\end{align*}
Now $\abs{(X+\tstep\,  Q_n(a(X,\cdot) ))^2\, Q_n(b(X,\cdot))^2} \le 2   K^2\, X^2 + 2\,\tstep ^2\, K^4$. Consequently,
\begin{align}\label{conseq}
\mean{e^{\alpha\, X_{n+1}^2}}%
&\le
{ e^{ \alpha\, X^2\,(1+c\, \tstep +c\,\tstep\, \alpha)%
		 +c\,\alpha\,\tstep +c\,\tstep ^3 \alpha^2}}.
\end{align}

\cref{alg} is used in the iteration, so that the support is reduced and Gauss quadrature is applied. Note that
\[
Q(e^{\alpha \,x^2})\le	\int_\real e^{\alpha\, x^2}\,\mu(dx),\qquad \alpha>0,
\]
where $Q$ is a Gauss quadrature rule for $\mu$ (by applying \cref{classic_gq} and noting that even derivatives of $e^{\alpha \,x^2}$ are non-negative).  Similarly, the support reduction moves mass inwards and the resulting integral of $e^{\alpha \,x^2}$ is reduced.  Consequently, if $X\sim Q_n$ in \cref{conseq}, we  have
\begin{align*}
Q_{n+1}(e^{\alpha\, x^2})%
&\le
Q_n(e^{\alpha\, (1+c \,\tstep +c\,\tstep \,\alpha) x^2})\,%
e^{c\,\alpha \,\tstep( 1+\tstep^2 \,\alpha)}.
\end{align*}
We can iterate this to find a bound on $Q_n(e^{\alpha\, x^2})$ in terms of $Q_0(e^{\alpha\, x^2})$. The value of $\alpha$ changes at each step of the iteration, and
\begin{align*}
Q_{n+1}(e^{\alpha_{0}\, x^2})%
&\le
Q_{n}( e^{ \alpha_1\, x^2} )\,%
 e^{c\,\alpha_0 \,\tstep\, (1+\tstep^2 \,\alpha_0)},
\end{align*}
where $\alpha_1=\alpha_{0}\,(1+c \,\tstep )+c\,\tstep\, \alpha_{0}^2$.

Let $\alpha_{n+1}=\alpha_n\,(1+c\,\tstep ) + c\,\tstep\, \alpha_n^2$.
If $\alpha_{n}\le1$, then $\alpha_{n+1}\le \alpha_n\,(1+2\,c\,\tstep )\le \alpha_0\,(1+2\,c\,\tstep )^n\le \alpha_0\, e^{2\,c\,t_n}$. We see that, if $\alpha_0 \le e^{-2\, c}$, then $\alpha_n\le e^{2\,c\, t_n} \alpha_0\le 1$ for $t_n\le 1$.  It is now easy to show that
\begin{align*}
Q_{n}( e^{\alpha_{0} \,x^2 } )%
&\le
Q_{n-m} (e^{ \alpha_m\, x^2}) \,e^{2\, c},
\end{align*}
for $t_n\le 1$ and any $\alpha_0\le e^{-2\,c}$. In particular, $Q_n(e^{\lambda \,x^2}) \le Q_0(e^{ \alpha \,x^2})\, e^{2\,c}$ for $\lambda\le e^{-2\,c} \min\{\alpha,1\}$.
\end{proof}%
We examine the error incurred reducing the support to $[-R,R]$.
\begin{lemma}\label[lemma]{supred}
	Let $\mu$ be a probability measure on $\real$ and suppose that $\mu(e^{\lambda \,x^2})<K$, for some $\lambda>0$. For $\tstep>0$, define the measure $\mu_\tstep$ by
	\[
	\mu_\tstep(A)%
	\coloneq \mu(A\cap 1_{ (-R,R) }) %
	+ \mu((-\infty,-R])\, \delta_{-R}(A)%
	+ \mu([R,\infty)) \,\delta_R(A),
	\]
	for
	$R=\sqrt{(4/\lambda)\,\abs{\log\tstep}}$ and Borel sets $A\subset \real$.
	There exists $c>0$, independent of $\tstep$, such that
	\[
	\abs{\mu(\phi)-\mu_\tstep(\phi) }%
	\le c \,\norm{\phi}_{0,\beta}\,%
	\tstep^2,\qquad %
	\forall \phi\in F^{0,\beta}.
	\]
\end{lemma}
\begin{proof}  It suffices to consider the two measures $\mu$ and $\mu_\tstep$ on the tail $(-R,R)^c$, as they are equal on $(-R,R)$. First, note that
	\[
	\abs{\mu(1_{(-R,R)^c} \,\phi)}%
	\le e^{-\lambda R^2/2}\,%
	\mu( \Phi),\qquad%
	\Phi(x)\coloneq  e^{\lambda \,x^2} \,e^{-\lambda \,x^2/2} \,1_{(-R,R)^c}(x)\, \abs{\phi(x)},
	\]
	where $1_{S}$ denotes the indicator function on the set $S$.
	As $\phi\in F^{0,\beta}$, $\abs{\phi(x)}\le \norm{\phi}_{0,\beta}\,(1+\abs{x}^\beta)$ and $\abs{e^{-\lambda \,x^2/2}\,\phi(x)}$ is uniformly bounded by $c\,\norm{\phi}_{0,\beta}$ for a constant $c$ independent of $R$ and $\phi$, but dependent on $\beta$ and $\lambda$. Hence,
	\[
	\abs{\mu(1_{(-R,R)^c} \,\phi) }%
	\le c\, \norm{\phi}_{0,\beta}\, e^{-\lambda\,R^2/2}\, \mu(e^{\lambda\, x^2} )%
	\le c \, \norm{\phi}_{0,\beta}\,e^{-\lambda\, R^2/2}.
	\]
	For $R=\sqrt{(4/\lambda)\,\abs{\log \tstep} }$, we see that $e^{-\lambda R^2/2} \le \tstep^2$. Hence, $\abs{\mu(1_{(-R,R)^c} \phi)}$ is bounded by $c\,\norm{\phi}_{0,\beta}\,\tstep^2$. The same applies to $\abs{\mu_{\tstep}(1_{(-R,R)^c}\phi)}$ by a similar argument  and the proof is complete.
\end{proof}
Thus, the support reduction with $R=\sqrt{(4/\lambda)\,\abs{\log \tstep}}$ maintains accuracy if $\mu(e^{\lambda x^2})$ is finite and the test function grows polynomially.
Next, we estimate the error for the Gauss quadrature at step $n$.

\begin{lemma}\label[lemma]{bounds} Suppose that $a,b\in C_K^0(\real^2)$. Let $Q_R(\cdot)=Q_{n+1}^\pm(\cdot \cap (-R,R))$ and let $Q$ be the $m_{n+1}$-point Gauss quadrature rule approximating $Q_R$.
	If  $m_{n+1}\ge m_{n}$, for all $\phi\in C^{2m_n}(\real)$,
	\[
	\abs{  Q_R(\phi)-Q(\phi)}%
	\le    K\, (2\,R)^{2m_{n}-1}  \,\tstep^{1/2}\,%
	\frac1 { (2m_{n})! }\,%
	\sup_{x\in(-R,R)}\abs{\phi^{ (2m_{n})}(x) }.
	\]
\end{lemma}
\begin{proof} If both $X_{n+1}^{j\pm}$ belong to $(-R,R)$, let $z^j=x^j+Q_n(a(x^j,\cdot))\,\tstep$. Then
\begin{equation}
\abs{ X_{n+1}^{j\pm}-z^j}%
\le \abs{Q_n( b ( x^j,\cdot) ) } \,\tstep^{1/2} %
\le K \,\tstep^{1/2}\label{this}
\end{equation}
and $\abs{z^j} \le R$ (every $z^j$ lies half way between $X^{j\pm}_{n+1}$). If only one $X_{n+1}^{j\pm} \in(-R,R)$, let $z^j$ be that point.
	 The measure $Q_R$ has at most $2m_n$ points and we apply \cref{c1}  with $N=2$  and $\delta=K \,\tstep^{1/2}$. In general, $Q_R$ may have less than $2m_n$ points and we should trivially extend $Q_R$ to apply \cref{c1} (i.e., extend $Q_R$ to a $2m_{n+1}$-point rule by adding zero-weighted points in $(-R,R)$ consistent with \eqref{this}).
\end{proof}
\begin{corollary} \label[corollary]{corr}%
	Let $a,b\in C^0_K(\real)$. Let $Q$ be the $m_{n+k}$-point Gauss quadrature rule for $Q_R(\cdot)=Q_{n+k}^\pm(\cdot \cap (-R,R))$ (i.e., after not performing \cref{alg} $(k-1)$-times).  Suppose that $m_{n+k}\ge m_n$. For each $k$, there exists $c>0$ such that,
	for all $\phi\in C^{ 2m_n}(\real)$,
	\[
	\abs{  Q_R(\phi)-Q(\phi)}%
	\le    c\, (2\,R)^{2 m_{n}-1}  \tstep^{1/2}%
	\frac1 { (2m_{n})! }\,%
		\sup_{x\in(-R,R)}\abs{\phi^{ (2m_{n})}(x) }.
	\]
\end{corollary}
\begin{proof}
	This is a simple extension of \cref{bounds} using \cref{c1}.
\end{proof}

\section{Error analysis for ordinary SDEs}\label{err}
The proposed algorithm has much in similarity to those
introduced by~\cite{Muller-Gronbach2015-vv}. In that paper, Ito--Taylor methods for a general class of
multi-dimensional SDEs are developed that use support-reduction strategies to improve efficiency. They reduce the support of the measure by reducing its diameter and eliminating points whilst maintaining
moment conditions. Along with a non-uniform time-stepping regime, the authors provide detailed error and
complexity analyses.
The present situation is similar and effectively we are
transplanting \cref{alg} for their
reduction strategies. Using appropriate Gauss quadrature error
estimates, much of their analysis applies in the present case.

The estimate in \cref{corr} depends on the radius $R$ of the support. We now choose $R=\sqrt{(4/\lambda)\,\abs{\log\tstep}}$, for $\lambda$ given by \cref{mf_mom}. Fix $k$ (the number of steps between applying \cref{alg}) and $\beta$ (to choose test functions $\phi\in F^{0,\beta}$).
\begin{proposition}\label[proposition]{t}
	Let $R=\sqrt{(4/\lambda)\,\abs{\log\tstep}}$ in \cref{alg}. Then,
	for all $\phi \in F^{0,\beta}\cap C^{2m_{n}}(\real)$,
	\begin{align*}
	\abs{  Q_{n+k}^{\pm}(\phi)-Q_{n+k}(\phi)}%
	&\le    c  \,\tstep^{1/2}\, \abs{\frac {16}\lambda\log \tstep}^{\frac{2m_n-1}2}%
	\frac1 { (2m_{n})! }\\%
	&\qquad \times\sup_{x\in(-R,R)}
	\abs{\phi^{ (2m_{n}) } (x) }%
	+c \, \norm{\phi}_{0,\beta}\,\tstep^2.
	\end{align*}
\end{proposition}
\begin{proof} The error due to the Gauss quadrature on $(-R,R)$ is described by \cref{corr}. Applying \cref{mf_mom} with \cref{supred}, the error due to the support reduction is bounded by $c \,\norm{\phi}_{0,\beta}\,\tstep^2$. Summing the two gives the desired upper bound.
\end{proof}

Given $\tstep>0$ and a $k\in\naturals$, we choose the number of points $m_n$ as the smallest non-negative integer such that
 \begin{equation}\label[ineq]{eq:MM}
 \log \Gamma(2m_n+1)\ge M_1(m_n, \tstep, n+k),\qquad \forall t_{n+k}<1,
 \end{equation}
where $\Gamma(x)$ denotes the gamma function and
\begin{equation}\label{eq:M}
M_p(m,\tstep, n)%
\coloneq \pp{p+\frac 12}\,\abs{\log \tstep}%
+\frac {2m-1}2 \, \log \abs{\frac{16}\lambda\log \tstep}
+ \pp{m-2}  \,\abs{ \log (1-t_n) }.
\end{equation}
We now describe how fast $m_n$ increases as $\tstep$ decreases. Assuming the Golub--Welsch algorithm takes $\order{m^3}$ operations, $\sum_n m_n^3$ gives the amount of work needed to apply \cref{alg} at every time step and we describe its growth.
\begin{theorem}\label{blimey}
	The number of Gauss quadrature points $m_n$ is a non-decreasing function of $n$. As the time step decreases, $m_n$ is non-decreasing. The number of points $m_n$ for $Q_n$ satisfies
	\[
	m_n%
	\le 1+\max\Bp{
	\frac 34 \,\abs{\log \tstep},
	\displaystyle \frac{e^2}{2} \,\sqrt{\frac{16\,\abs{\log{\tstep}} }{\lambda\,(1-t_n)} } }.
	\]
	In particular, $\sum_{t_n<1} m_n^3 =\order{(\abs{\log \tstep}\,/\,\tstep)^{3/2}}$.
\end{theorem}
\begin{proof}
The function $M_1$ is increasing in $n$ (via $t_n\in(0,1)$) for $m\ge 2$. Hence, $m_n$ is  non-decreasing in $n$ ($m_n$ is discrete and may not change as $t_n$ is varied by small amounts).
Also, for fixed $t_n$, $M_1$ is a decreasing function of $\tstep$, and hence $m_n$ is non-decreasing as $\tstep$ decreases.
From \cref{eq:M},
\[
M_1(m,\tstep,n)\le \frac32 \,\abs{\log \tstep}%
+{m}\, \log \frac{16\,\abs{\log \tstep}} {\lambda\,(1-t_n)}.
\]
	Stirling's formula \cite[Eq. 6.1.37]{Abram} tells us that
	\[
	x!%
	=  \Gamma(x+1)%
	= \sqrt{2\pi} \,x^{x+1/2}\,\exp\pp{-x+\frac{\theta}{12x}},\qquad \text{for some $\theta\in(0,1)$,}
	\]
	and hence $\log\abs{(2m)!}=\log \Gamma(2m+1) \ge \frac 12\, \log(4\,\pi\, m)+ 2\,m\,(\log(2\,m)-1)$. Then,
	\[
	\log \Gamma(2\,m+1) %
	\ge 2\,m+ 2\,m\,\pp{\log(2\,m)-2}%
	=2\,m+m \log \frac{4\,m^2}{e^4}.
		\]
		If $m\ge (3/4)\abs{\log \tstep}$ and $m\ge (e^2/2) \sqrt{16\,\abs{\log \tstep}/(\lambda\,(1-t_n))}$, then $\log \Gamma(2\,m+1)\ge M_1(m,\tstep,n)$.  Hence, as $\sum_{k=1}^\infty k^{-3/2}$ is finite,
		\[
		\sum_{t_n<1} m_n^3%
		\le c \,\pp{\frac{16\,\abs{\log \tstep}}{\lambda\,\tstep}}^{3/2}.
		\]
\end{proof}

We now give the main convergence theorem for ordinary SDEs. In this case, the coefficients $a(x,y)$ and $b(x,y)$ are independent of the mean-field $y$. We  choose the single-point initial distribution $\mu=\delta_x$ and write $X(t)$ for $X^{\mu}(t)$ and $P^x_t$ for $P^{\delta_x}_t$.

\begin{assumption}\label[assumption]{ass}
	Suppose that $K\ge 1\ge \lambda>0$ and assume that
	$x_0\in[-K,K]$ and $a,b\in C^{4}_K(\real)$ and $b^2(x)\ge
	\lambda$ for all $x\in\real$.
\end{assumption}

\begin{theorem} \label{sodemain}Let  \cref{ass} hold. Consider
  the $m_n$-point Gauss quadrature rule $Q_n$
  defined in \cref{alg} with $m_n$ given by \cref{eq:MM} and $R=\sqrt{(4/\lambda)\abs{\log \tstep} }$.
 The total error satisfies
  \[
  \abs{P_1^x(\phi) -Q_N(\phi)}%
  \le \begin{cases}
   c\, \norm{\phi}_{2,\beta}\,%
  (1+\abs{x}^c)\,
    \,\tstep\,\abs{\log \tstep},  \qquad \forall \phi\in F^{2,\beta},
    \\[0.5em]
     c\, \norm{\phi}_{3,\beta}\,%
     (1+\abs{x}^c)\,
     \,\tstep,  \qquad\quad\;\qquad \forall \phi\in F^{3,\beta},
    \end{cases}
  \]
  for a constant $c$ independent of $K$.
\end{theorem}
\begin{proof} Let $g_n(x)\coloneq P^{x}_{1-t_n} (\phi)\equiv \mean{\phi(X(1-t_n))\,|\, X(0)=x}$ for $x\in\real$. Notice that $g_N=\phi$ and $g_0(x)=P^x_1(\phi)$. Let $T_\tstep(\phi)(x)=\mean{\phi(\Psi(x,\tstep,\cdot))}$ for $\Psi$ defined in \cref{psi}.
 The total error
\[
P^x_1(\phi)
-Q_N(\phi)%
=\sum_{n=1}^N E_n^T %
+\sum_{n=1}^N E_n^G,%
\]
where $E_n^G=Q_n^\pm (g_n)-Q_n (g_n)$ (the error
due to \cref{alg})  and $E_n^T= Q_{n-1}(g_{n-1})-Q_n^\pm (g_n)= Q_{n-1} (P_\tstep^x (g_n))-Q_{n-1}(T_\tstep(g_n))$ (the bias error due to Euler--Maruyama over time step $\tstep$). We estimate the two sources of error, focusing on the case where $\phi \in F^{2,\beta}$.

	Local truncation error: Under \cref{ass}, \cite[Eq. (35) with  $\gamma=1$]{Muller-Gronbach2015-vv} shows that
	$E_n^T$ satisfies
	\[
	\abs{E_n^T}%
	\le \begin{cases}\displaystyle
	c\, \norm{\phi}_{4,\beta}\, \pp{1+\abs{x}^c}\,
	\frac{\tstep^2}{1-t_n},& n=1,\cdots,N-1,\\[0.7em]
		c\, \norm{\phi}_{2,\beta}\, \pp{1+\abs{x}^c}\,
		\tstep,&n=N.
	\end{cases}
	\]

	\cref{alg} error: We do not apply \cref{alg} on the final step and so $E_N^G=0$. For $n=1,\dots,N-1,$
	 \cref{t} gives that
	\begin{align*}
	\abs{E_n^G}%
	=\abs{ Q_{n}^\pm (g_n)-Q_n (g_n) }%
	&\le c\, \norm{g^{(2 m_{n-k} ) }_{n}}_\infty %
	\frac{1}{(2	m_{n-k})!}\,  %
	\tstep^{1/2}
	\abs{\frac {16}\lambda \log \tstep}^{\frac{2 m_{n-k}-1}2}\\%
	&\qquad+c\,\norm{\phi}_{0,\beta}\,\tstep^2.%
	\end{align*}
	\cite[Lemma 8]{Muller-Gronbach2015-vv} provides that
	\[
	\norm{g_n}_{k,\beta} %
	\le \norm{\phi}_{2,\beta}
	\frac{1}{(1-t_n)^{(k-2)/2}},\qquad \forall k\ge 4.
	\]
	Consequently,
	\[
	\abs{E_n^G}%
	\le c\,%
	\norm{\phi}_{2,\beta}\,%
	\frac{1}{(1-t_n)^{ m_{n-k} -1  } }\,%
	\frac{1}{(2m_{n-k})!}\,
	\tstep^{1/2}\,%
	\abs{\frac {16}\lambda\log \tstep}^{ \frac{2 m_{n-k}-1}2} %
	+c\,\norm{\phi}_{0,\beta}\,\tstep^2.
	\]
	Notice that
	\[
	\frac{1}{(1-t_n)^{ m_{n-k}-1 } }%
	\frac{1}{(2m_{n-k})!}
	\tstep^{1/2 }%
	\abs{\frac {16} \lambda \log \tstep}^{ \frac{2 m_{n-k}-1}2}
	\le\frac1{(1-t_n)}
	\tstep^2,%
	\]
	if
	\begin{align*}
	{\Gamma(2m_{n-k} +1)}%
	&\ge
	\tstep^{-3/2}\,%
	{	\abs{\frac {16} \lambda \log \tstep}^{ \frac{2 m_{n-k}-1}2}} %
	\frac{1}{(1-t_n)^{m_{n-k}-2 } }.
	\end{align*}
	This holds  as we have chosen $m_{n-k}$ satisfies
	$
	\log \Gamma(2m_{n-k}+1) %
	\ge M_1(m_{n-k},\tstep, n),
	$
	for $M_1$ defined in \cref{eq:M}.
	Then, $\abs{E^G_n} \le c\,(\norm{\phi}_{2,\beta}+1)\,\tstep^2/ (1-t_n)$.

Summing all the errors and using $\sum_{n=1}^{N-1} \tstep/(1-t_n) \le \log (N)=\abs{\log \tstep}$, we complete the proof. For $\phi\in F^{3,\beta}$, the argument is similar except the $(1-t_n)$ factors do not arise and so the $\abs{\log \tstep}$ term does not appear.
 \end{proof}

\section{Error analysis for mean-field SDEs}\label[section]{err2}

We now generalise our error analysis to mean-field SDEs. We wish to show that $Q_n$ approximates $P_{t_n}^\mu$, starting from a good approximation of the initial distribution, $Q_0\approx\mu$. To express the closeness of $Q_0$ to $\mu$, we use the Wasserstein distance. For any probability measures $\mu,\nu$ on $\real$, define  the Wasserstein distance
\[
W_{k,\beta}(\mu,\nu)%
\coloneq \sup\Bp{\abs{\mu(\phi)-\nu(\phi)}%
	\colon \norm{\phi}_{k,\beta}\le 1}.%
\]

\begin{assumption}\label[assumption]{ass:initial}
	The initial measure $Q_0$  satisfies $Q_0(e^{\alpha\,x^2})<\infty$ for some $\alpha>0$ independent of $\tstep$ and approximates $\mu$ in the sense that $W_{2,\beta}(\mu,Q_0)\le c\,\tstep$.
\end{assumption}

Under this assumption, \cref{mf_mom} applies and $Q_n(e^{\lambda \,x^2})$, for $t_n\le 1$, is uniformly bounded for some $\lambda>0$.  We choose $R=\sqrt{(4\,/\lambda)\,\abs{\log \tstep}}$ in \cref{alg}.

We introduce a non-autonomous SDE corresponding to the mean-field SDE with $P^\mu_t(a(X,\cdot))$ and  $P^\mu_t(b(X,\cdot))$ treated as known functions of $(X,t)$.
Let $X(t;s,x)$ for $t\ge s$ denote the
solution of
\begin{equation}
dX%
=\bar{a}(X,t)\,dt%
+ \bar{b}(X,t) \,dW(t),\qquad
X(s;s,x)=x,\label{eq:mf_sde2}
\end{equation}
for $\bar{a}(X,t)\coloneq P_t^\mu(a(X,\cdot))$ and $\bar{b}(X,t)\coloneq P_t^\mu(b(X,\cdot))$.
Here we fix the initial distribution as a delta measure at
$x$ and keep the same measure $P_t^\mu$ from \cref{eq:mf_sde} for the mean fields.
Note that $\int_{\real} \mean{ \phi( X(t;0,x) ) }\,\mu(dx)=P_t^\mu(\phi)$, so that $P^\mu_t(\phi)=\mu (P_{0,t} (\phi) )$ for $P_{s,t}(\phi)(x)\coloneq \mean{ \phi( X(t;s,x) ) }$. In this notation,  we drop the
$\mu$ superscript, even though the non-autonomous SDE depends on $\mu$ via the drift and diffusion.

In the following assumption on the drift and diffusion, the mean-field diffusion $\bar b$ is used to set a non-degeneracy condition.

\begin{assumption}\label[assumption]{ass2}
	Suppose that $a,b\in C^4_K(\real^2)$ and, for some $K\ge 1\ge \lambda>0$, that $\bar{b}^2(t,x)\ge
	\lambda$  for $x\in\real$ and $t\in[0,1]$.
\end{assumption}
The main theorem for the numerical approximation of mean-field SDEs by GQ1 is the following. The method of selecting the number of Gauss points $m_n$  is modified to approximate the distribution uniformly on the time interval. In this case, $m_n\equiv m$ should be chosen independent of $n$. We choose $m$ as the smallest integer greater than the initial number of points $m_0$ such that  $
\log \Gamma(2m+1)\ge M_1(m, \tstep, n+k)$ where $M_1$ is given by
\begin{equation}\label{eq:MMM}
M^{\text{mf}}_p(m,\tstep, n)%
\coloneq \pp{p+m-\frac32}\,\abs{\log \tstep}%
+\frac {2m-1}2\, \log \abs{\frac {16} \lambda\log \tstep}
\end{equation}
or
\begin{equation}\label{eq:MMM1}
M^{\text{smooth}}_p(m,\tstep, n)%
\coloneq \pp{p+\frac 12}\,\abs{\log \tstep}%
+\frac {2m-1}2\, \log \abs{\frac {16} \lambda\log \tstep}.
\end{equation}
The choice of $M_1$ depends on the regularity of the underlying problem, as described in \cref{thm:main_mf}. The time $t_n$ appears on the right-hand side in neither case and $m$ is independent of $n$. In the following, the overall work for the time-stepping is dominated by $\sum_{t_n\le 1} m_n^3$ (the work to compute the Gauss quadrature rule at each step). The work to compute the initial measure $Q_0$ is often neglible, for example, if the initial distribution is Gaussian or in other cases where accurate quadrature rules are easily computed.

\begin{theorem}\label{blimey2}
	Denote the initial number of points for the rule $Q_0$ by $m_0$. For \cref{eq:MMM},
	\[
	m%
	\le%
	\max\Bp{m_0,1+\frac{e^2}{2} \,\sqrt{\frac{8\,\abs{\log{\tstep}} }{\lambda\, \tstep} }}.
	\]
	If the work to compute $Q_0$ is $\order{\smash{\abs{\log \tstep}^{3/2}/\tstep^{5/2}}}$ and the initial number of points $m_0=\order{\smash{\abs{\log \tstep}^{1/2}/\tstep^{1/2}}}$, then the overall total work $\order{\smash{\abs{\log \tstep}^{3/2}/\tstep^{5/2}}}$.
	For \cref{eq:MMM1},  	\[
	m%
	\le\max\Bp{m_0,1+\frac34 \abs{\log \tstep},1+\frac{e^2}{2} \,\sqrt{\frac{8\,\abs{\log{\tstep}} }{\lambda} }}.
	\]
		If the work to compute $Q_0$ is $\order{\abs{\log \tstep}^3/\tstep}$ and the initial number of points $m_0=\order{\abs{\log \tstep}}$, then the overall total work $\order{\abs{\log \tstep}^3/\tstep}$.
\end{theorem}
\begin{proof}
	From \cref{eq:MMM},
	\[
	M_1^{\mathrm{mf}}(m,\tstep,n)\le {m}\, \log \frac{16\,\abs{\log \tstep}} {\lambda\,\tstep}
	\]
and
	\[
	\log \Gamma(2\,m+1) %
	\ge 2\,m+m \log \frac{4\,m^2}{e^4}.
	\]
	If $m\ge (e^2/2) \sqrt{16\,\abs{\log \tstep}/(\lambda\,\tstep)}$, then we have $\log \Gamma(2\,m+1)\ge M_1^{\mathrm{mf}}(m,\tstep,n)$. Similarly, from \cref{eq:MMM1},
		\[
		M_1^{\mathrm{smooth}}(m,\tstep,n)\le \frac32 \,\abs{\log \tstep}%
		+{m}\, \log \frac{16\,\abs{\log \tstep}} {\lambda}.
		\]	If $m\ge (3/4)\abs{\log \tstep}$ and $m\ge (e^2/2) \sqrt{16\,\abs{\log \tstep}/\lambda}$, then we see $\log \Gamma(2\,m+1)\ge M_1^{\mathrm{smooth}}(m,\tstep,n)$.
				 The estimate for the total work follows as $\sum_{n=1}^N m^3=m^3/\tstep$.
\end{proof}
In the following, we show upper bounds on the error for smooth and rough problems, and smooth in this case indicates infinite differentiability, which is much stronger than in \cref{sodemain}. This is because infinite differentiability  allows the reduction of the number of Gauss points $m$ to $\order{\abs{\log \tstep}^{1/2}}$ from $(\abs{\log \tstep}/\tstep)^{1/2}$.
\begin{theorem}\label{thm:main_mf}
	Let \cref{ass1,ass:initial,ass2} hold and the number of Gauss points $m$ be given by \cref{eq:MMM}.
	For some $c>0$
	\[
	\max_{t_N\le 1}
\abs{ P^\mu_{t_N}(\phi)%
		-Q_N(\phi)}
	\le c\,%
	\norm{\phi}_{2,\beta}\,%
	\tstep\,\abs{\log \tstep}%
	,\qquad %
	\forall \phi\in F^{2,\beta}.
	\]
	If in addition to \cref{ass1}, we have $W_{\infty,\beta}(\mu,Q_0)\le c\,\tstep$ and in addition to \cref{ass2}, we have $a,b\in C^\infty_K(\real^2)$, and the number of Gauss points $m$ is given by \cref{eq:MMM1}, then
	\[
		\max_{t_N\le 1}
		\abs{ P^\mu_{t_N}(\phi)%
			-Q_N(\phi)}
		\le c\,%
		\norm{\phi}_{\infty,\beta}\,%
		\tstep%
		,\qquad %
		\forall \phi\in F^{\infty,\beta}.
		\]
\end{theorem}
Before the proof, we develop a sequence of lemmas. First, we show that the Euler--Maruyama step depends continuously on the initial measure $\mu$ in terms of the Wasserstein distance.
\begin{lemma}\label[lemma]{lemma:b}
	Suppose that $a,b\in C^k_K(\real^{2})$.
	There exists $c>0$ such that, for any $x\in\real$,
	\begin{equation}\label[ineq]{imp}
	\abs{\delta(x)} %
	\le%
	\begin{cases}
	 c\,\tstep\, %
	\norm{g}_{3,\beta}\, %
	(1+\abs{x}^\beta)\,%
	W_{k,\beta}(\mu,\nu),& \quad\forall g\in F^{3,\beta},\\[1em]
		 c\,\tstep\, %
		 \norm{g}_{2,\beta}\, %
		 (1+\abs{x}^\beta)\,%
		 (W_{k,\beta}(\mu,\nu)+1),& \quad\forall g\in F^{2,\beta},
		  \end{cases}
	\end{equation}
	where
	$
	\delta(x)%
	\coloneq\mean{g(\Psi(x, \tstep,\mu)) }%
	-\mean{g(\Psi(x,\tstep,\nu)}$ and $\Psi$ is defined by \cref{psi}.
\end{lemma}
\begin{proof}
	Let $x_{\lambda,\mu}=x%
	+\lambda \,\mu (a(x,\cdot))\,\tstep%
	+(\lambda \,\mu (b(x,\cdot))\,\xi\,\sqrt{\tstep}$ and
	\begin{align*}
	\phi(\lambda;g)%
	&=
	g\pp{x_{\lambda,\mu}}-g(x_{\lambda,\nu}).
	\end{align*}
	Then $\delta=\mean{\phi(1;g)}$ and $\phi(0;g)=0$  and
	\begin{align*}
	\phi'(\lambda;g)%
	&=g'(x_{\lambda,\mu}) \bp{\mu (a(x,\cdot))\,\tstep+ \mu (b(x,\cdot)) \,\sqrt{\tstep}\,\xi}\\
&\qquad-g'(x_{\lambda,\nu})\,\bp{\nu(a(x,\cdot))\,\tstep+ \nu( b(x,\cdot)) \,\sqrt{\tstep}\,\xi}.
	\end{align*}
	Note that $\mean{\phi'(0;g)}=g'(x) (\mu-\nu) (a(x,\cdot))\tstep$ as $\mean{\xi}=0$.
	By Taylor's theorem,
	\begin{align*}
	\delta&=\mean{\phi(0;g)%
		+\phi'(0;g)%
		+\int_0^1 \phi''(\lambda;g)\,\lambda\,d\lambda}\\%
	&=	g^\prime(x) \,\tstep\, %
	{(\mu-\nu) (a(x,\cdot)) }%
	+\mean{\int_0^1 \phi''(\lambda;g)\, \lambda \,d\lambda}.
	\end{align*}
	Now,
	\begin{align*}
	&\abs{\phi''(\lambda;g)}\\%
	&\le \abs{g''(x_{\lambda,\mu})} \pp{ \bp{\mu(a(x,\cdot))\tstep+\mu(b(x,\cdot))\sqrt{\tstep}\xi}^2%
	-\bp{\nu(a(x,\cdot))\tstep+\nu(b(x,\cdot))\sqrt{\tstep}\xi}^2}\\
	&\qquad +\abs{g''(x_{\lambda,\mu})-g''(x_{\lambda,\nu}) } \cdot
	\abs{ \nu(a(x,\cdot))\,\tstep+\nu(b(x,\cdot))\,\sqrt{\tstep}\,\xi}^2.
	\end{align*}
	Hence, as $a,b,\xi$ are all bounded,
	\[
	\abs{\delta} \le c\,(1+\abs{x}^\beta)\, W_{k,\beta}(\mu,\nu)\,\pp{\norm{g}_{0,\beta} \, \tstep + \norm{g}_{2,\beta}\, \tstep+ \norm{g}_{3,\beta} \,\tstep}.
	\]
	This now implies the first equation in \cref{imp}. The second is similar.
\end{proof}
\begin{lemma} \label[lemma]{bbar}%
	Let \cref{ass1,ass2} hold. If $a,b\in C^k_K(\real^2)$, then $\bar{a}$ and $\bar{b}$ belong to $C^k_K(\real^2)$.
\end{lemma}
\begin{proof} Under \cref{ass1}, $P_t^\mu$ has a smooth density and $\bar{a}, \bar{b}$ inherit their smoothness from $a$, $b$, and the density. The argument is given in more detail in \cite[page 431]{Antonelli2002-er}.
\end{proof}
\begin{lemma}\label[lemma]{reg_g}
	Let \cref{ass1,ass2} hold and $g_{n,N}\coloneq P_{t_n,t_N}\phi$. Then,  for non-negative integers $r,k$,
	\begin{equation*}
	\norm{g_{n,N}}_{k,\beta}%
	\le c\, \norm{\phi}_{r,\beta}\,
	\frac{1}{(t_N-t_n)^{(k-\min\{k,r\})/2}},\qquad \forall \phi \in F^{k,\beta}.
	\end{equation*}
\end{lemma}
\begin{proof} For the autonomous case, see \cite[Lemma 8]{Muller-Gronbach2015-vv}.  In this case, the drift and diffusion are non-autonomous. The argument generalises as \cite[Chapter 9, Theorem 7]{Friedman2013-ud} applies also for time-dependent coefficients with the assumptions given.
\end{proof}

The next lemma states a bound on the  local truncation error.
\begin{lemma} \label[lemma]{lemma:a}
	Let \cref{ass1,ass2} hold.
	There exists $c>0$ such that
	\[
	\abs{P_{t_{n-1},t_n}(\phi)(x)%
		- \mean{\Psi (x, \tstep, P_{t_{n-1}}^\mu) } }%
	\le \begin{cases}
	c \,%
	\norm{\phi}_{4,\beta}\,%
	(1+\abs{x}^c)\,%
	\tstep^{2},&\qquad%
	\forall \phi\in F^{4,\beta},\\[1em]
	c \,%
	\norm{\phi}_{2,\beta}\,%
	(1+\abs{x}^c)\,%
	\tstep,&\qquad%
	\forall \phi\in F^{2,\beta}.
	\end{cases}
	\]
\end{lemma}

\begin{proof} When $a,b$ are independent of the second argument, this is implied by
	\cite[Lemma 3 with $\gamma=1$]{Muller-Gronbach2015-vv}. In our case, the drift is $\bar a(X, t)$ and diffusion $\bar b(X,t)$, which are smooth functions according to \cref{bbar} and their lemma is easily extended.
\end{proof}

\begin{proof}[Proof of \cref{thm:main_mf}]
	Define the measure $e_N=P_{t_N}^\mu-Q_N$ and consider $\phi\in F^{2,\beta}$. Let $g_{n,N}\coloneq P_{t_n,t_N}(\phi)$, so that $g_{n,n}=\phi$. Decompose the error $e_N(\phi)$ for $N\ge 1$ as
	\begin{equation}\label{err4}
	e_{N}(\phi)%
	=\sum_{n=1}^N E_n^{T_1} + E_n^{T_2} + E_n^G ,
	\end{equation}
	where $E_n^{T_1}$ represents the error from the Euler--Maruyama discretisation of the non-autonomous system, $E_n^{T_2}$ represents the error from the mean-field, and $E_n^G$ represents the error from \cref{alg} applied to $g_{n,N}$. In detail,  let
		\begin{align*}
		\mathsf{I}%
		&\coloneq  Q_{n-1}\, ( P_{t_{n-1},t_n}( P_{t_n,t_N}(\phi)) ) %
		=\int_\real P_{t_{n-1},t_N}( g_{n,N})(x)\,Q_{n-1}(dx),\\%
		\mathsf{II}%
		&\coloneq
		\int_\real \mean{g_{n,N}(\Psi(x, \tstep, P_{t_{n-1}}^\mu))}\,Q_{n-1}(dx),\\%
		\mathsf{III}%
		&\coloneq
		Q^\pm_{n}  ( P_{t_{n}, t_N} (\phi) )
		=\int_\real \mean{g_{n,N}(\Psi(x, \tstep, Q_{n-1})}\,Q_{n-1}(dx),\\
		\mathsf{IV}%
		&\coloneq
		Q_{n} ( P_{t_{n},t_N} (\phi) ),
		\end{align*}
	where $\mean{\cdot}$ denotes the expectation over $\xi$ in the definition of $\Psi$ (see \cref{psi}). Consider the telescoping sum
	\[
	e_{N}(\phi)%
	=\sum_{n=1}^N \pp{Q_{n-1}(P_{t_{n-1},t_N}(\phi) ) %
		- Q_{n} (P_{t_n,t_N} (\phi) ) \strutB  }.
	\]
	We have \cref{err4}  for $E_n^{T_1}=\mathsf{I}-\mathsf{II}$, $E_n^{T_2}=\mathsf{II}-\mathsf{III}$, $E_n^G=\mathsf{III}-\mathsf{IV}$. We estimate the three sources of error in turn. We focus on the rough case (i.e., $\phi\in F^{2,\beta}$) and briefly note the differences with the smooth case.

{Local truncation error for non-autonomous SDE:} From \cref{lemma:a}, with $n<N$,
	\begin{align*}
	\abs{\mathsf{I}-\mathsf{II}}%
	&=
\abs{	Q_{n-1}  \pp{ P_{t_{n-1},t_n} ( g_{n,N})(x) %
		-\mean{g_{n,N}(\Psi(x, \tstep, P^\mu_{t_{n-1}}))} } }\\
	&\le c\,\norm{g_{n,N}}_{4,\beta} %
	\bp{1+Q_{n-1}(\abs{x}^c)\strutB} %
	\tstep^{ 2 }.
	\end{align*}
	By \cref{mf_mom}, $Q_n(\abs{x}^c)$ is uniformly bounded and, by \cref{reg_g}, $\norm{g_{n,N}}_{4,\beta}$ is bounded by $c\,\norm{\phi}_{2,\beta}/(t_N-t_n)$.  Similarly, for $n=N$,  $	\abs{\mathsf{I}-\mathsf{II}} \le c \norm{\phi}_{2,\beta}(1+Q_{n-1}(\abs{x}^c)) \tstep$. Hence, $\sum_{n=1}^N \abs{ \mathsf{I}-\mathsf{II} } \le c\, \norm{\phi}_{2,\beta} \,\tstep \,\abs{\log \tstep}$.
		In the smooth case, the estimate is the same, without the $(t_N-t_n)$ singularity and hence without the $\log$ term.

 	{Mean-field error:}  From \cref{lemma:b},
	\begin{align*}
\abs{	\mathsf{II}-\mathsf{III} }%
	&\le
	\abs{ Q_{n-1}\pp{%
	 \mean{ g_{n,N}( \Psi(x,\tstep,  P^\mu_{t_{n-1}} ) ) } %
	-  \mean{ g_{n,N}(\Psi(x,\tstep ,Q_{n-1})) }  } }\\
	&\le c%
	\pp{1+Q_{n-1}(\abs{x}^\beta) }\,%
	\tstep\,%
	\norm{g_{n,N}}_{3,\beta}\,%
	W_{4,\beta}(P^\mu_{t_{n-1}}, Q_{n-1}).
	\end{align*}
 By \cref{mf_mom}, $Q_n(\abs{x}^\beta)$ is uniformly bounded and, by \cref{reg_g}, $\norm{g_{n,N}}_{3,\beta}$ is bounded by $c\norm{\phi}_{2,\beta}/(t_N-t_n)^{1/2}$ for $n=1,\dots,N-1$.  Hence, \[\abs{\mathsf{II}-\mathsf{III} } \le c\, \tstep\, \norm{\phi}_{2,\beta}\,W_{2,\beta}(P^\mu_{t_{n-1}},Q_{n-1})\frac{1}{(t_N-t_n)^{1/2}}.\]  For $n=N$, \[
 \abs{\mathsf{II}-\mathsf{III} } \le K \,\pp{ 1+ Q_{N-1}(x^\beta) }\, \tstep\, %
 	 \norm{\phi}_{2,\beta}\,%
 	  W_{2,\beta} (P^\mu_{t_{N-1}}, Q_{N-1} ) + \norm{\phi}_{2,\beta}\, K\,\tstep.\]
 	  In the smooth case, $\phi\in F^{\infty,\beta}$ and $a,b\in C^{\infty}_K(\real^2)$, so that $\norm{g_{n,N}}_{3,\beta}$ is uniformly bounded and $\abs{\mathsf{II}-\mathsf{III}} \le  c\, \tstep\, \norm{\phi}_{\infty,\beta}\,W_{\infty,\beta}(P^\mu_{t_{n-1}},Q_{n-1})$.
	\item 	{\cref{alg} error:} We consider the case where \cref{alg} is applied at every step $n=1,\dots,N-1$.  Then, for each $n$,
	\[
	\mathsf{III}-\mathsf{IV}%
	= Q^\pm_{n} (g_{n,N})-Q_n(g_{n,N}) .%
	\]
	Here $Q_{n}$ is the measure given by approximating $Q^\pm_{n}$ by \cref{alg} and the associated error is described by \cref{t}. Thus, recalling that
	$R=\sqrt{(4/\lambda)\,\abs{\log\tstep}}$,
	\[
	\abs{\mathsf{III}-\mathsf{IV}}
	\le
	 c\, (2R)^{2 \,m_n-1}\,  \tstep^{1/2}\,%
	 \frac1 { (2\,m_n)! }\,%
	 \norm{g_{n,N}^{ (2\,m_n) }} _\infty + c\norm{\phi}_{0,\beta}\tstep^2.
	\]
Applying  \cref{reg_g},
	\begin{align*}
	\abs{\mathsf{III}-\mathsf{IV}}
	&\le
	c\, (2R)^{2\, m-1}\,  %
	\tstep^{1/2}\,
	\frac1 { (2\,m)! }%
	\, \norm{\phi}_{2,\beta}\,
	\frac{1}{(t_N-t_n)^{m-1}}%
  + c\norm{\phi}_{0,\beta}\tstep^2\\%
	&\le
	c\, (2R)^{2\, m-1}\,  %
	\frac1 { (2\,m)! }%
	\, \norm{\phi}_{2,\beta}\,
	\frac{1}{\tstep^{m-5/2}} \frac1{t_N-t_n}+ c\norm{\phi}_{0,\beta}\tstep^2.
	\end{align*}
	This is bounded by $c \norm{\phi}_{2,\beta}\tstep^2/(t_N-t_n)$ if
	$\log \Gamma(2\,m+1)\ge M_1^{\mathrm{mf}}(m, \tstep, n+k)$ for $M_1^{\mathrm{mf}}$ defined by \cref{eq:MMM}.

	In the smooth case, $\abs{\mathsf{III}-\mathsf{IV}}\le c \tstep^2 \norm{\phi}_{\infty,\beta}$  if
	$\log \Gamma(2\,m+1)\ge M_1^{\mathrm{smooth}}(m, \tstep, n+k)$ for $M_1^{\mathrm{smooth}}$ defined by \cref{eq:MMM1}.
	Sum the three upper bounds to show that
		\begin{align*}
		e_N(\phi)%
		&\le c\,\norm{\phi}_{2,\beta}\,\tstep\,\abs{\log \tstep}+c\,\norm{\phi}_{2,\beta}\,
		\sum_{n=1}^{N-1}%
			{\frac{\tstep}{(t_N-t_n)^{1/2} } W_{2, \beta}(P^\mu_{t_{n-1}},Q_{n-1})}%
			\\&\qquad+ c \,\norm{\phi}_{2,\beta}\, \tstep\,W_{2,\beta}(P^\mu_{t_{N-1}}, Q_{N-1}),\qquad t_N\le 1.
		\end{align*}
		Take the supremum over $\phi \in F^{2,\beta}$,
		\begin{align*}
	&W_{2,\beta}(P^{\mu}_{t_{N}}, Q_{N})\\%
	&\le
	c\,\tstep\,\abs{\log \tstep}+%
	\sum_{n=1}^{N-1}{%
	\frac{\tstep}{(t_N-t_n)^{1/2}}\, W_{2,\beta}(P^\mu_{t_{n-1}},Q_{n-1})} + c \,\tstep\,W_{2,\beta}(P^\mu_{t_{N-1}}, Q_{N-1}).
	\end{align*}
	We assume that $W_{2,\beta}(P^\mu_0, Q_0) \le c\,\tstep$ in \cref{ass:initial}.
		Gronwall's inequality completes the proof of the rough case.
		In the smooth case, similar arguments show that
	\begin{align*}
	W_{\infty,\beta}(P^{\mu}_{t_{N}}, Q_{N})%
	&\le
	c\,\tstep+%
	\sum_{n=1}^N{%
		{\tstep} \,W_{\infty,\beta}(P^\mu_{t_{n-1}},Q_{n-1})}
	\end{align*}
	and Gronwall's inequality again gives the result.
\end{proof}

Consider \cref{eq:mf_sde_ref}, where a nonlinear dependence on the time-$t$ distribution is allowed via functions $A,B\colon\real\to\real$. Our numerical method generalises by replacing the definition of $\Psi$ in \cref{psi} with
\begin{equation}
  \Psi(x,\tstep,Q)%
  \coloneq x%
  + \tstep\,A( Q (a(x,\cdot)))%
  +\sqrt{\tstep}\, B( Q( b(x,\cdot)))\,\xi.
\end{equation}
Gauss quadrature can be used in the same way with the same choice of $m_n$ and the same estimates apply as long as $A,B$ have regularity consistent with \cref{lemma:b,lemma:a}. This leads to the following convergence and complexity result.
\begin{corollary}
  Let \cref{ass1,ass:initial,ass2} hold and $A,B \in C_K^k(\real^d)$. Let the number of Gauss points $m$ be given by \cref{eq:MMM} and  $P^\mu_t$ be the solution of \cref{eq:mf_sde_ref} with initial distribution $\mu$. Then, for some $c>0$,
  \[
  \max_{t_N\le 1}
  \abs{ P^\mu_{t_N}(\phi)%
    -Q_N(\phi)}
  \le c\,%
  \norm{\phi}_{2,\beta}\,%
  \tstep\,\abs{\log \tstep}%
  ,\qquad %
  \forall \phi\in F^{2,\beta}.
  \]
  If $Q_0$ is cheap to compute (see \cref{thm:main_mf}) and $m_0=\order{(\abs{\log(\tstep)}/\tstep)^{1/2}}$, the total work is   $\order{\abs{\log \tstep}^{3/2}/\tstep^{5/2}}$.
  If in addition to \cref{ass1}, we have $W_{\infty,\beta}(\mu,Q_0)\le c\,\tstep$ and in addition to \cref{ass2}, we have $a,b\in C^\infty_K(\real^2)$ and $A,B\in C^\infty_K(\real)$, and the number of Gauss points $m$ is given by \cref{eq:MMM1}, then
  \[
    \max_{t_N\le 1}
    \abs{ P^\mu_{t_N}(\phi)%
      -Q_N(\phi)}
    \le c\,%
    \norm{\phi}_{\infty,\beta}\,%
    \tstep%
    ,\qquad %
    \forall \phi\in F^{\infty,\beta}.
    \]
		  If $Q_0$ is cheap to compute  and $m_0=\order{\abs{\log \tstep}}$, the total work is   $\order{\abs{\log \tstep}^{3}/\tstep}$.
\end{corollary}

\section{Numerical experiments}\label{num}

We now present a set of numerical experiments, exhibiting the behaviour of GQ1 as described in \cref{sec:alg}. We also try two methods that converge with second order.

	\textbf{GQ1e} The Richardson or Talay--Tubaro extrapolation involves taking two first-order approximations $P(\tstep)$ and $P(\tstep/2)$ of a quantity $P$, and computing $\hat P\coloneq 2\,P(\tstep/2)-P(\tstep)$. If $P$ has a second-order Taylor expansion,  $\hat{P}$ is a second-order accurate approximation to $P$. In the case that $P$ is generated by GQ1, this is very simple to code and implement and is included in the experiments. Thus, we define GQ1e to be the quadrature rule $Q$ defined by $2Q^{\tstep/2}-Q^{\tstep}$, where $Q^\tstep$ is the result of applying GQ1 with time step $\tstep$. The method results in a quadrature with some negative weights, which can lead to non-physical results when used with highly oscillatory $\phi$ and the method should be used with caution.

\textbf{GQ2}
Suppose that the mean-field SDE has the following structure
\begin{equation}\label{factor}
dX^\mu(t)%
=a(X^\mu(t),P_t^\mu(r)) \,dt%
+ b(X^\mu(t), P_t^\mu(r) )  \,dW(t)
\end{equation}
for given functions $a,b\colon \real\times \real^d\to\real$ and $r\colon \real\to\real^d$. Mean-field SDEs of this type, involving moments of the solution in the coefficient functions or vectors of monomials $r(x)=[x,x^2,\dots,x^d]$, were introduced in \cite{peter2} for example. By working out the second-order Ito--Taylor expansion, the following generalisation, which we name GQ2, of the Euler--Maruyama-based method GQ1 can be derived: let $\Delta W= \tstep\, \xi$ for $\xi$ given by three-point distribution with $\prob{\xi=0}=2/3$ and $\prob{\xi=\pm \sqrt 3}=1/6$ (i.e., the three-point Gauss--Hermite rule for $\Nrm(0,1)$). For a given measure $Q_n$, define $Q_{n+1}$ as the distribution of $X_{n+1}$ given by
\begin{align*}
	X_{n+1}%
	&=X+a \,\tstep+ b\,\Delta W%
	+\frac12 \partial_1 b\, b\,(\Delta W^2-\tstep)\\
	&\qquad+\frac12\pp{\partial_1 a\, b+\nabla b\cdot \mathcal{L} a+\frac 12 \partial_{11} b \,b^2}\, \Delta W \,\tstep\\%
	&\qquad+\frac 12 \pp{\nabla a\cdot \mathcal{L}a+\frac12\partial a_{11}\,b^2}\,\tstep ^2
\end{align*}
for
\[
\mathcal{L}a%
\coloneq\bp{a,%
	Q_n\pp{\partial_1 r\,	{ a} +\frac 12 \partial_{11} r\, b^2},\dots,%
	Q_n\pp{\partial_d r\,{ a} +\frac 12 \partial_{dd} r\, b^2}
},
\]
where $X\sim Q_n$ (independent of $\xi$) and all functions $a,b$ are evaluated at $(X, Q_n(r))$.
Here, $\partial_i$ and $\partial_{ii}$ denotes the first- and second-derivatives with respect to the $i$th argument, $\nabla a$ denotes the usual gradient in $\real^{d+1}$, and $\cdot$ the $\real^{d+1}$ inner product.

Though we do not include it, GQ2 submits to similar techniques of error analysis to GQ1. We expect second-order convergence in the Wasserstein distance $W_{4,\beta}$, so that test functions require two extra derivatives compared to GQ1.  The equation for the number of Gauss points $m_n$ needs to be adjusted by taking $p=2$ in \eqref{eq:MM},\eqref{eq:MMM}, or \eqref{eq:MMM1} as appropriate.
The total work for a given accuracy $\varepsilon$ is given by replacing $\tstep$ replaced by $\varepsilon^{1/2}$ in \cref{blimey,blimey2} (and increasing the regularity by two for all coefficients). For smooth mean-field equations, the work is $\order{\abs{\log \varepsilon}^{3} \varepsilon^{-1/2}}$.

We expect second-order convergence for both these method and the initial distribution $Q_0$ should be chosen with
$W_{4,\beta}(\mu,Q_0)\le c\tstep^2$.


The code for running these experiments is available for download
 \cite{sdelab}.

\subsection{Geometric Brownian motion}\label{example:gbm}
We consider the ordinary SDE for geometric Brownian motion given by
\[
dX(t)=\alpha\, X(t)\,dt + \sigma\, X(t) \,dW(t),\qquad %
X(0)=x,
\]
for parameters $\alpha,\sigma$ and initial data $x$.
For $\alpha=-1$, $\sigma=0.5$, and $x=1$, the exact value $\mean{X(1)}=e^{-1}$. We use this as a test case to compare with the multilevel Monte Carlo (MLMC) method, as in \cite[Example 8.49]{book}. The CPU time is compared against error, averaging over ten runs of MLMC to reduce the variance. The CPU time for the MLMC Matlab implementation (provided in \cite{book}) is scaled to match GQ1 at the first data point. See \cref{gq_gbm}. The errors for the Gauss quadrature methods are decaying at a much faster rate as the CPU time is increased. Theoretically, for a smooth problem like this, the work to achieve accuracy $\varepsilon$ for GQ1 behaves like $\varepsilon^{-1}\abs{\log \varepsilon}^{3}$, for GQ1e and GQ2 like $\varepsilon^{-1/2}\abs{\log \varepsilon}^{3}$, and for MLMC   like $\varepsilon^{-2}$. This is observed in the figure. Notice however that the linearly growing coefficients do not satisfy our assumptions.

\begin{figure}
	\centering
	\includegraphics[width=0.45\textwidth]{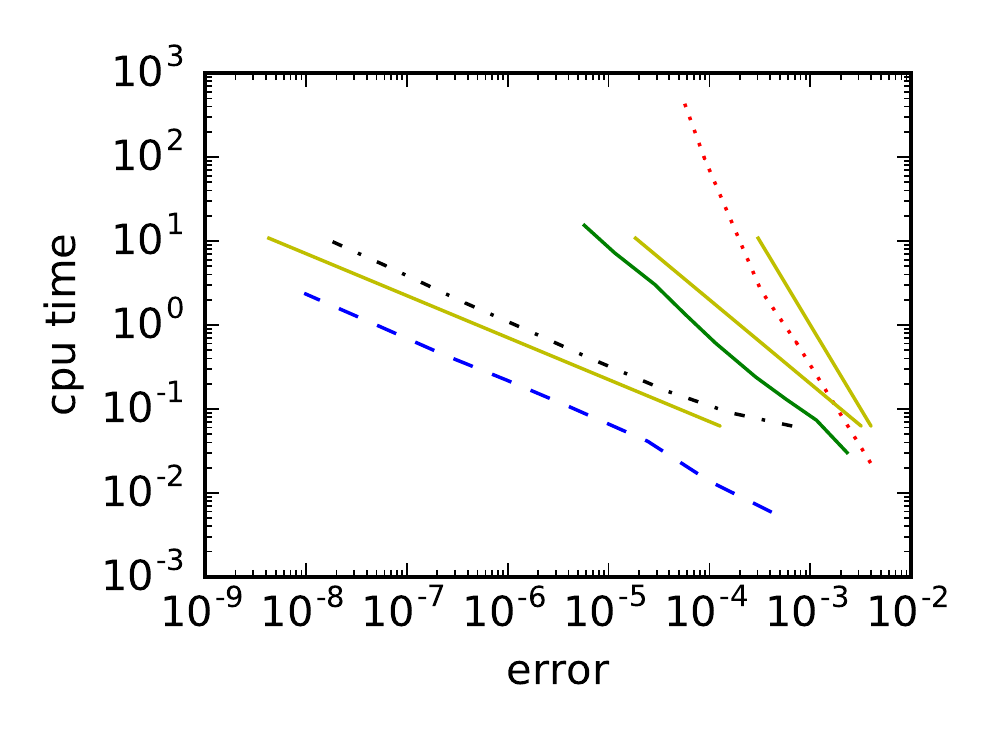}
	\caption{Geometric Brownian motion: The green line shows GQ1; the blue-dashed line shows GQ1e; the black-dash-dot line shows GQ2;  the red-dotted line shows MLMC. The cpu time for MLMC is scaled to match GQ1 at the first data point. Errors are computed relative to the exact value. The yellows lines indicate reference slopes of $-2$, $-1$, and $-1/2$.} \label{gq_gbm}
\end{figure}

\subsection{Generalised Ornstein--Uhlenbeck process} 
Consider the following generalisation of the Ornstein--Uhlenbeck SDE to a linear mean-field SDE:
\[
dX(t)=\bp{\alpha \,X(t)+ \beta \,\mean{X(t)}\strutB}\,dt + \sigma \,dW(t),\qquad X(0)=x,
\]
for parameters $\alpha,\beta,\sigma\in\real$ and initial data $x\in\real$.
By using Ito's formula, its first two moments can easily be calculated as
\begin{equation}\label{gbm_exact}
\mean{X(t)}=x\, e^{(\alpha+\beta)\,t},\qquad
\mean{X(t)^2}=x^2\,e^{2\,(\alpha+\beta)\,t } + \frac{\sigma^2}{2\,\alpha} \bp{e^{2\,\alpha\,t}-1}.
\end{equation}
It is used as a test case in \cite{Ricketson2015-xv},
with $\alpha=-1/2$, $\beta=4/5$, $\sigma^2=1/2$, $x=1$. We use these parameters and the results are  shown in \cref{eg2}. First-order convergence is observed for the first and second moments for GQ1,  and second-order convergence is observed  for both GQ1e and GQ2. The work is proportional to $\varepsilon^{-1}$ and $\varepsilon^{-1/2}$, reflecting the estimates (up to $\log$ terms) for smooth problems in \cref{blimey2}.

\begin{figure}
	\centering
	\includegraphics[width=0.45\textwidth]{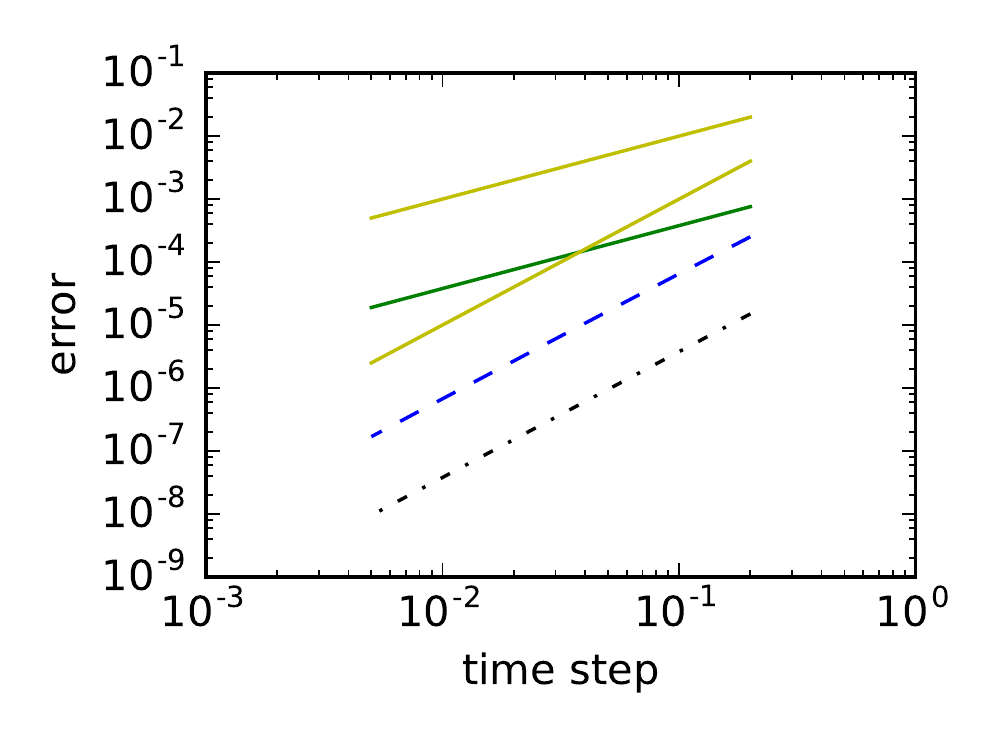}
	\includegraphics[width=0.45\textwidth]{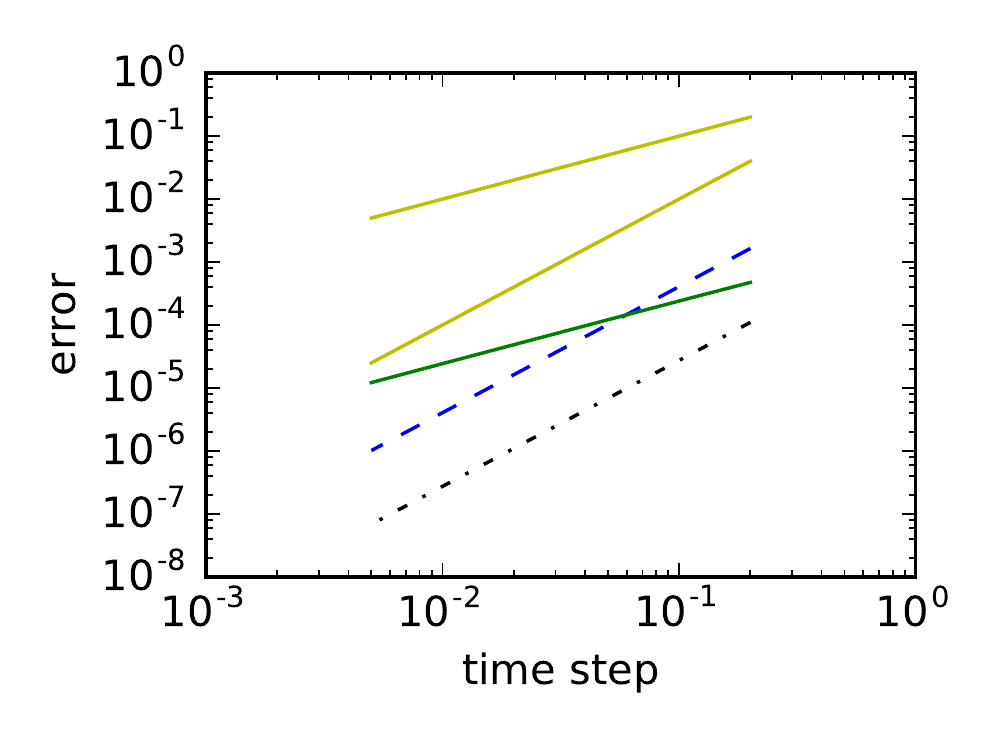}\\
	\includegraphics[width=0.45\textwidth]{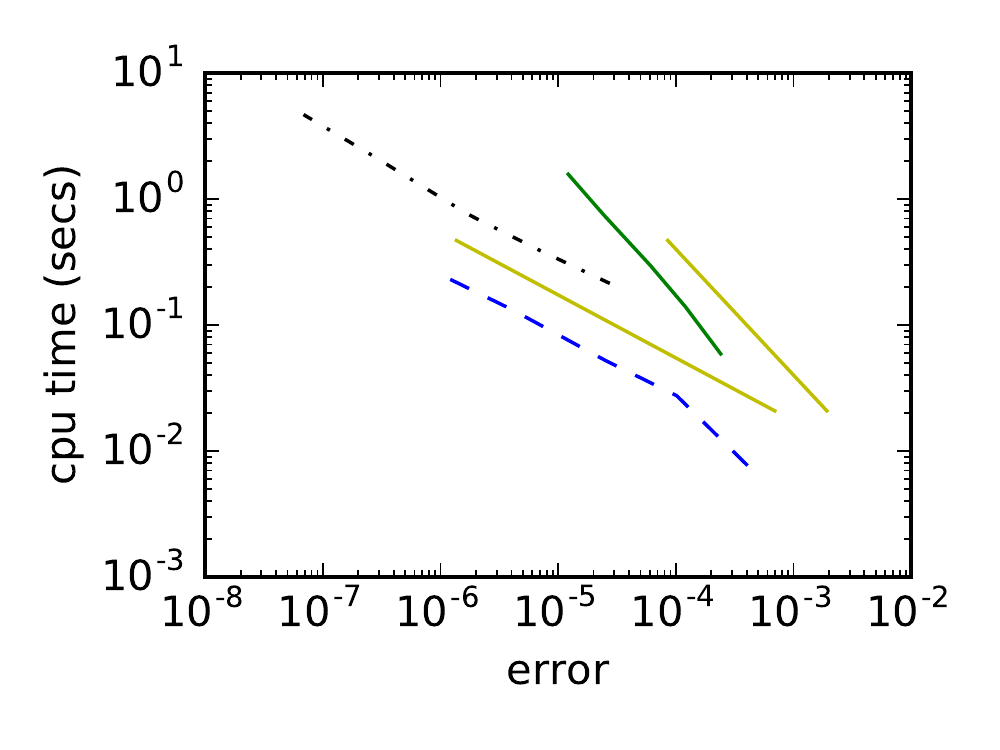}
	\includegraphics[width=0.45\textwidth]{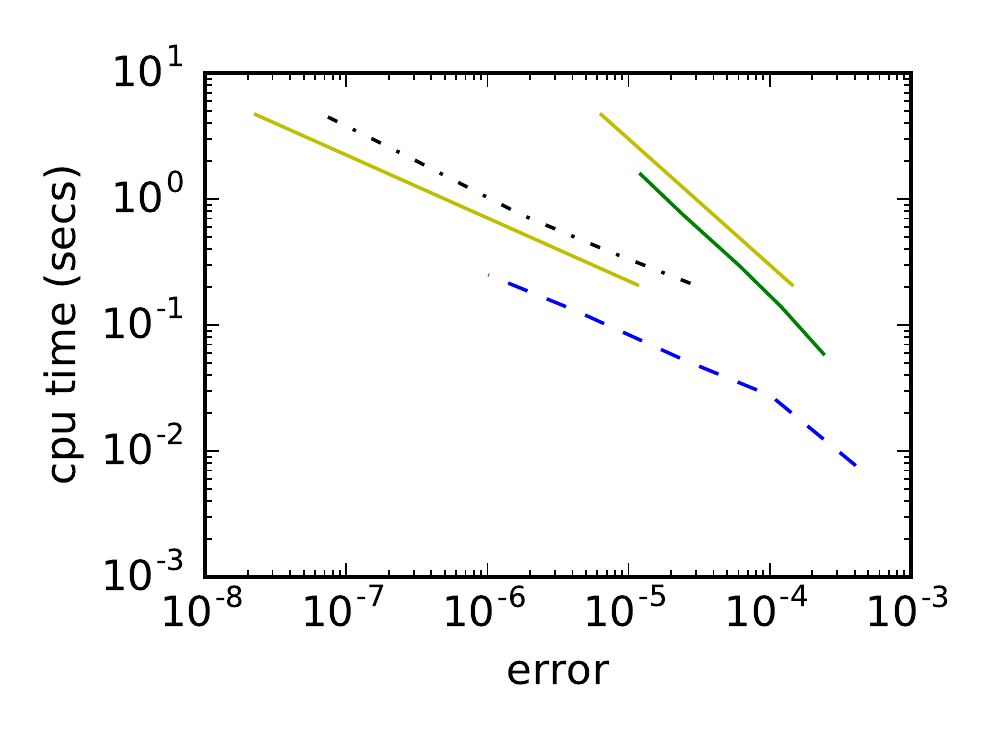}
	\caption{Generalised Ornstein--Uhlenbeck SDE: The green line shows GQ1; the blue-dashed line GQ1e; the black-dash-dot line shows GQ2. The yellow lines show reference slopes of $1$ and $2$ (top) and $-1/2$ and $-1$ (bottom) . The upper left- (resp., right-) hand plot shows the error in computing the mean (resp., second moment). The error is computed using reference values provided by \cref{gbm_exact}. The bottom plots shows the cpu time in seconds. } \label{eg2}
\end{figure}

\subsection{Polynomial drift}
The following mean-field Ito SDE
\begin{equation}\label{polydrift}
dX(t)%
=\bp{\alpha \,X(t)+\mean{X(t) }-X(t)\, \mean{X(t)^2}\strutB}\,dt%
+ X(t)\,dW(t),\qquad X(0)=x,
\end{equation}
for a parameter $\alpha\in\real$,
is considered in \cite{peter1}, where the first two moments of
$X(t)$ are shown to satisfy the system of ODEs
\begin{gather}\label{ode}\begin{split}
\frac{d\mean{X}}{dt}%
&=(\alpha+1)\,\mean{X}-\mean{X}\,\mean{X^2}\\
\frac{d\mean{X^2}}{dt}%
&=(2\,\alpha+1)\,\mean{X^2}+2\,\bp{\mean X}^2 - 2\,\bp{\mean {X^2}}^2,
\end{split}\end{gather}
with initial conditions $\mean{X}=x$ and $\mean{X^2}=x^2$.
We use this as a test with $\alpha=2$ and $x=1$ and results are shown in \cref{eg1}. Again first-order (GQ1) and second-order (GQ1e and GQ2) convergence is observed for the first and second moments and the cpu times behave in line with \cref{blimey2}.
\begin{figure}
	\centering
	\includegraphics[width=0.45\textwidth]{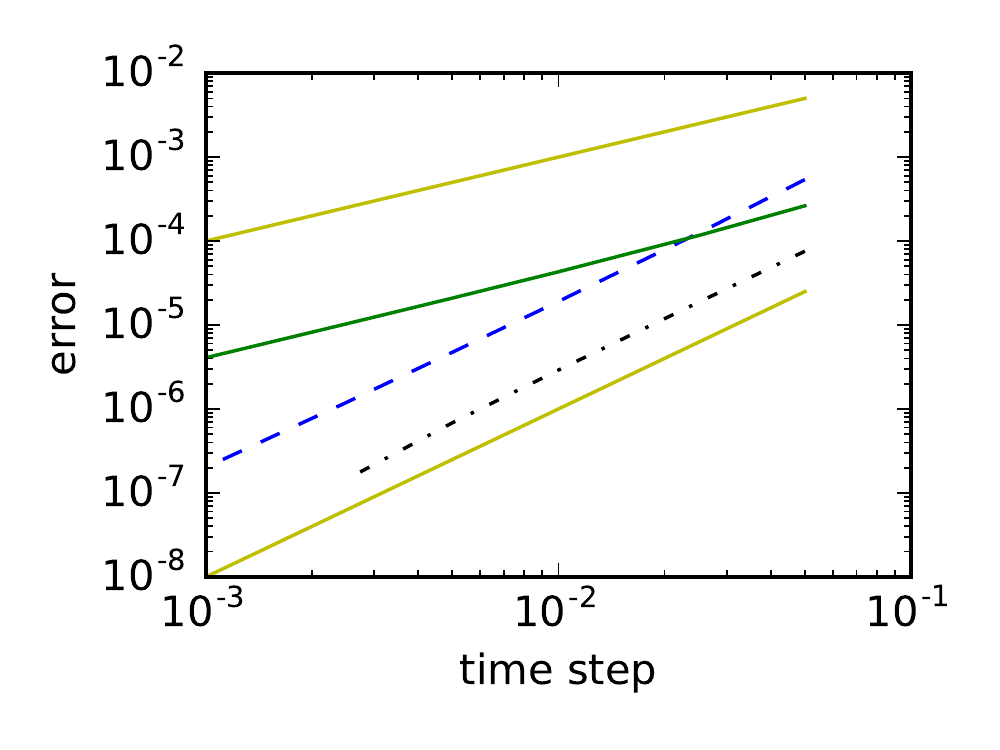}
	\includegraphics[width=0.45\textwidth]{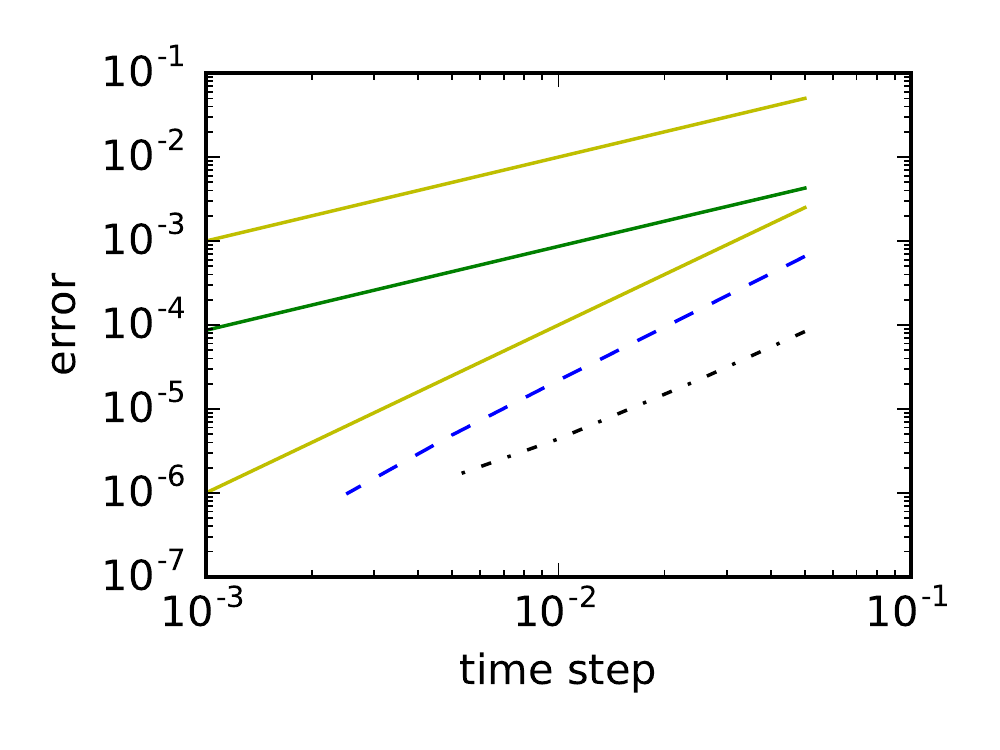}\\
	\includegraphics[width=0.45\textwidth]{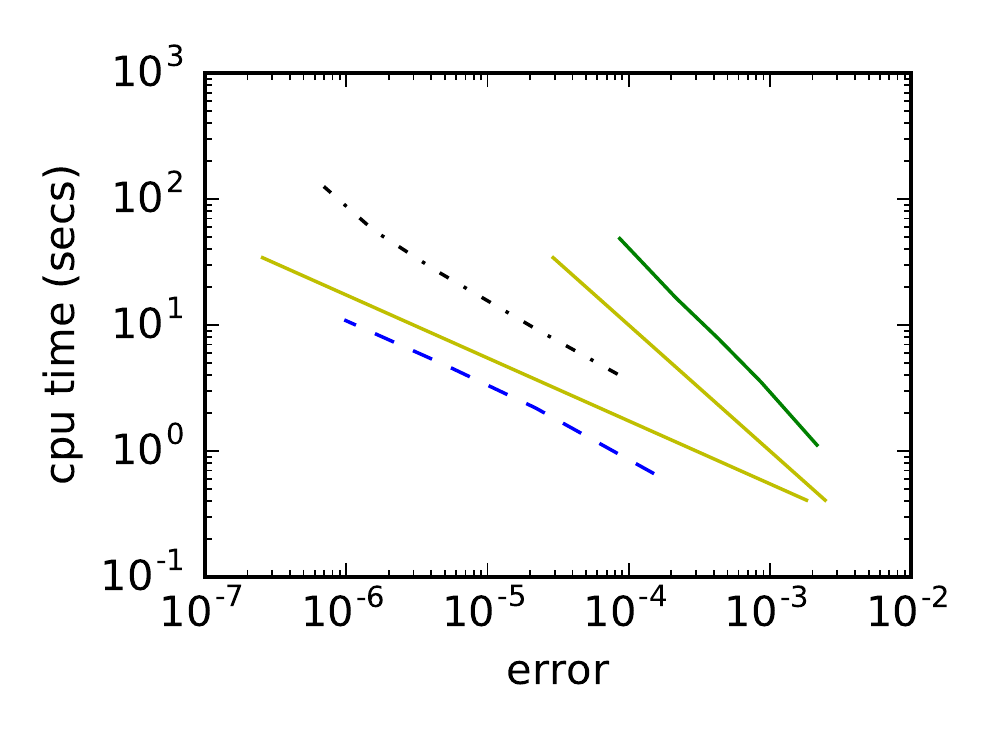}
	\includegraphics[width=0.45\textwidth]{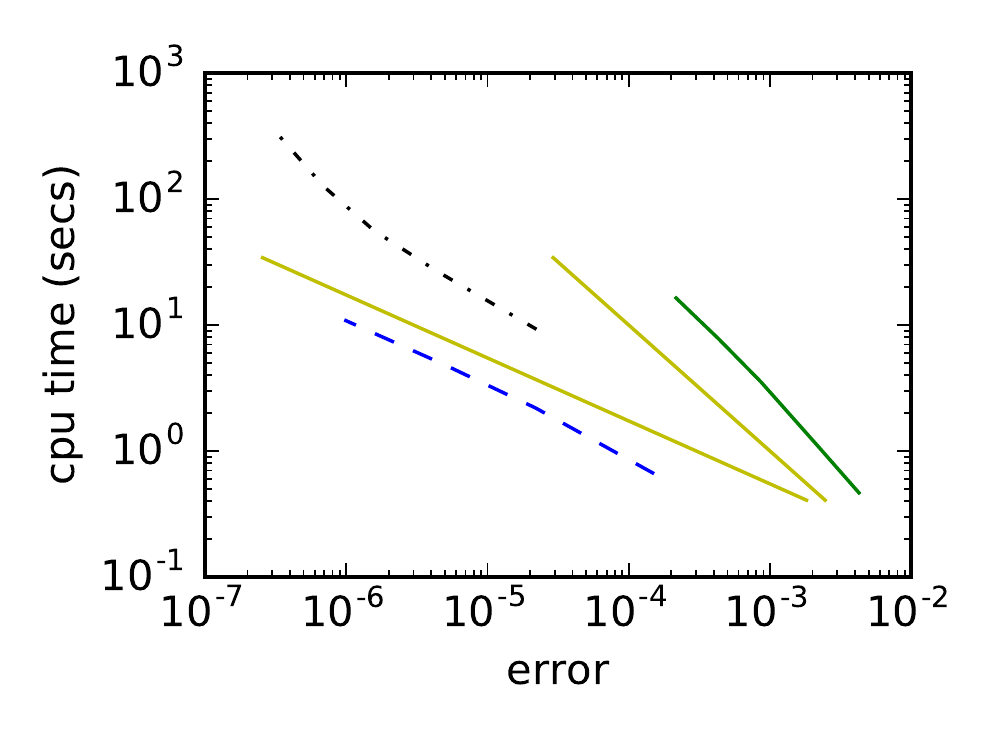}
	\caption{Polynomial drift: As for \cref{eg2} with the mean-field SDE \eqref{polydrift}. The error is computed by using an accurate numerical solution of \cref{ode} as a reference value.} \label{eg1}
\end{figure}
\subsection{Plane rotator} \label{sseg3}
The following is a model for coupled oscillators \cite{Kostur2002-nv} in the presence of noise:
\begin{equation}\label{plane}
dX^\mu(t)=\bp{ K \int_\real\sin(y-X^\mu(t))\,P^\mu_t(dy)-\sin(X^\mu(t))}\,dt + \sqrt{2\, k_B T}\,dW(t),
\end{equation}
for coupling parameter $K>0$, temperature $k_B T$, and initial condition $X^\mu(0)\sim \mu=\Nrm(\mu_0,\sigma_0^2)$.
In this case, we have a Gaussian initial distribution $\mu$, which can be approximated by Gauss--Hermite quadrature. The associated points and weights can be found tabulated or computed via the three-term recursion for the Hermite polynomials. In the implementation, we take the latter strategy and start with $Q_0$ equal to the $40$-point Gauss--Hermite rule.

 The variable $X^\mu(t)$ represents an angle. In place of the the diameter reduction step in \cref{alg} , we shift each point modulo $2\pi$ into $[0,2\pi)$. Also, we partition $[0,2\pi)$ into ten sub-intervals and apply Gauss quadrature on sub-intervals of width $L=\pi/5$. This significantly improves performance in experiments.

Following \cite{Ricketson2015-xv}, we choose parameter values
for $K=1$, $k_B T=1/8$ and initial mean $\mu_0=\pi/4$ and variance $\sigma^2_0=3\pi/4$.  Results are shown in \cref{eg3b}, which show errors for $P^\mu_1(\phi)$ for the test functions $\phi(x)=\sin^2(x)$ and $\phi(x)=\sin(x)$. Errors are computed by taking a reference solution given by GQ2.
First-order convergence is observed for GQ1 and second-order convergence is observed for GQ2.  The methods work rapidly and the finest solution has 434 quadrature points.
In \cref{cpdf}, we show the pdf and cdf of the initial and final distribution.

\begin{figure}
	\centering
	\includegraphics[width=0.45\textwidth]{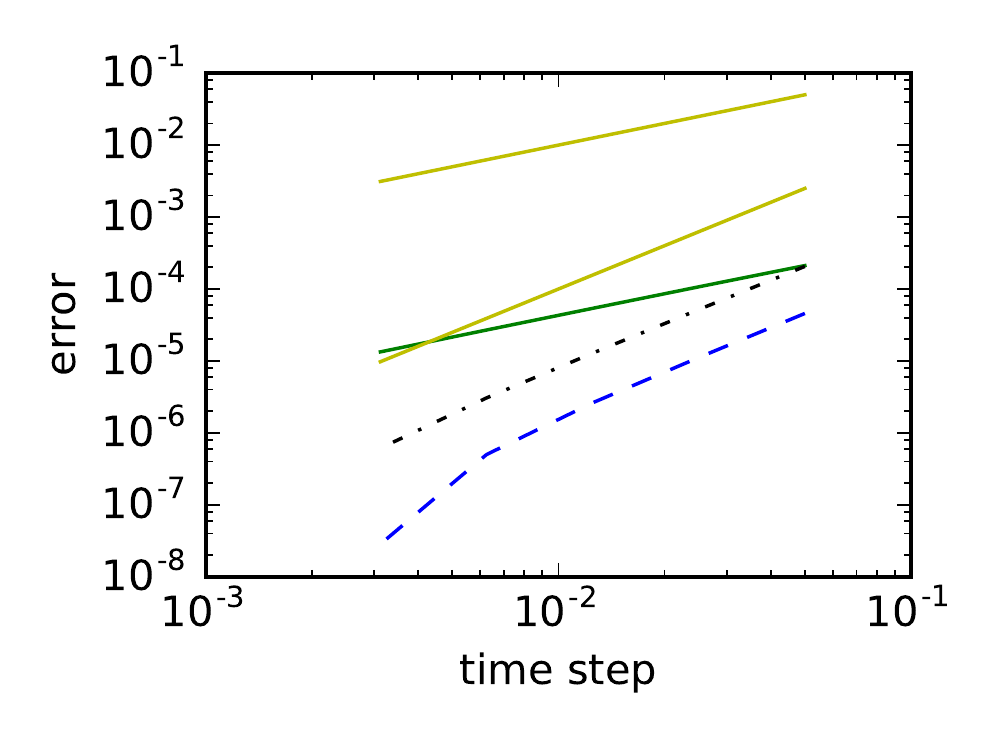}
	\includegraphics[width=0.45\textwidth]{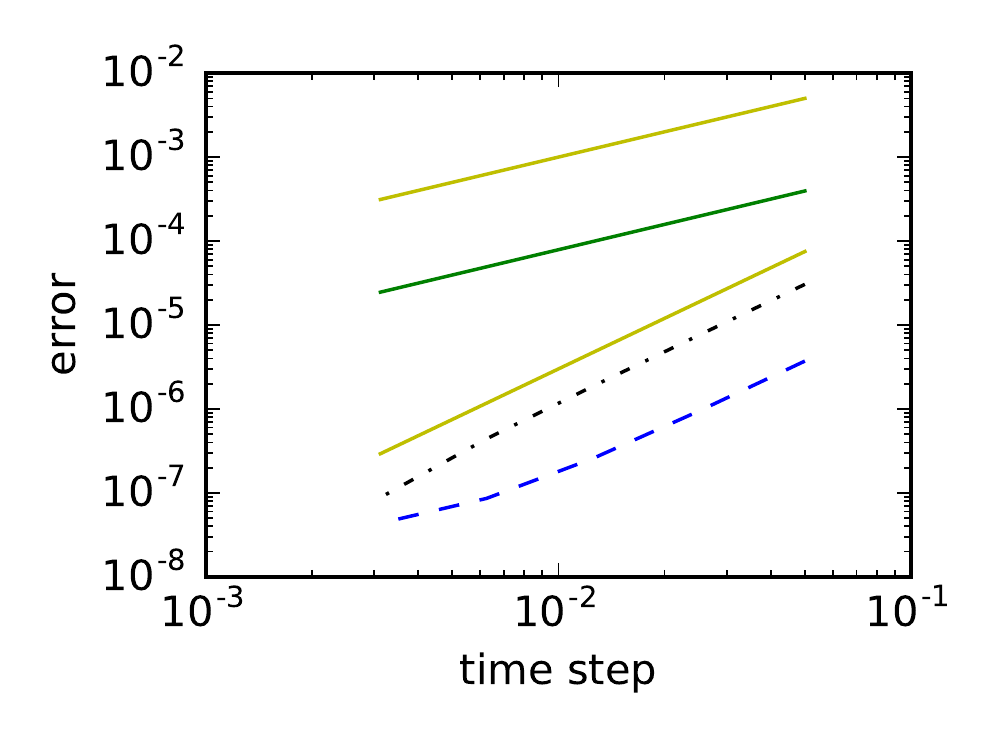}  \includegraphics[width=0.45\textwidth]{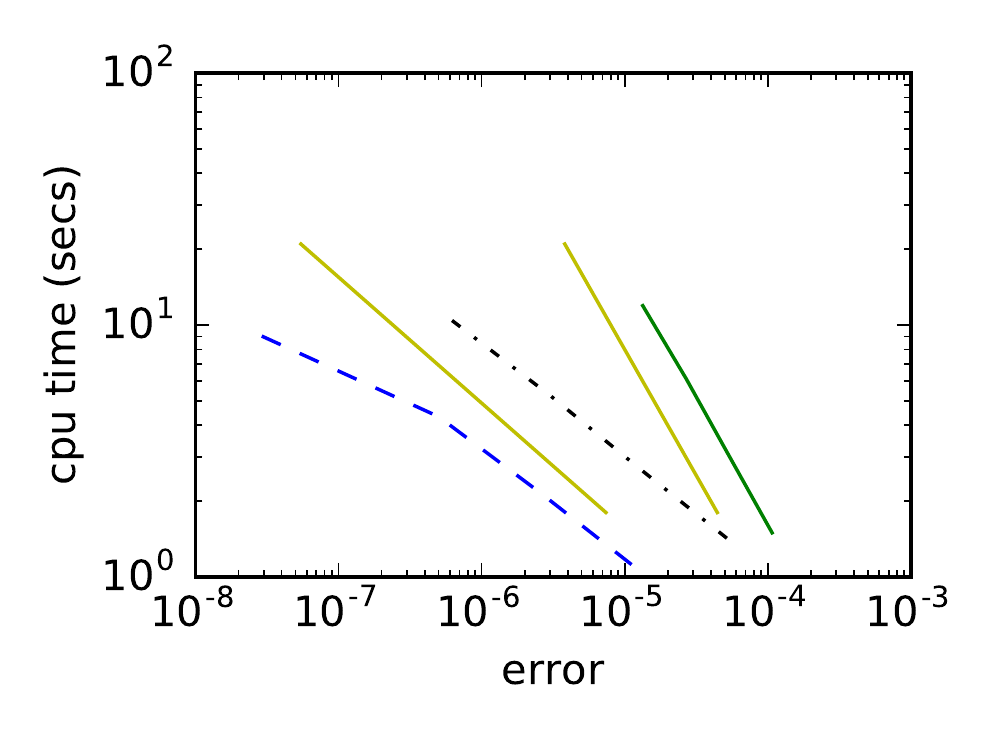}
	\includegraphics[width=0.45\textwidth]{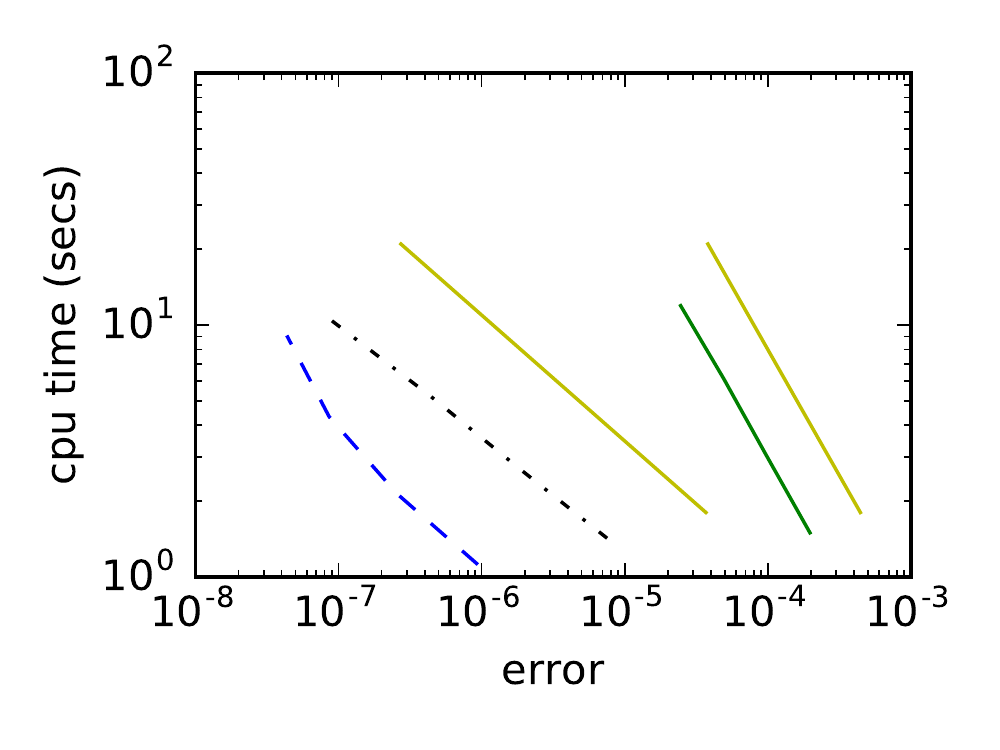}
	\caption{Plane rotator: error  against time step and cpu time for computing $\mean{\phi(X(1)) }$ for $\phi(x)=\sin^2(x)$ (left) and $=\sin(x)$ (right), via GQ1 (green), QG1e (blue dashed), and GQ2 (black dash-dot) methods for \cref{plane}. The yellow lines in the upper plots show slopes of $1$ and $2$, similar to the theoretical rate. The error is computed by taking a well-resolved GQ2 calculation for the reference value. } \label{eg3b}
\end{figure}

\begin{figure}
	\centering
	\includegraphics[width=0.45\textwidth]{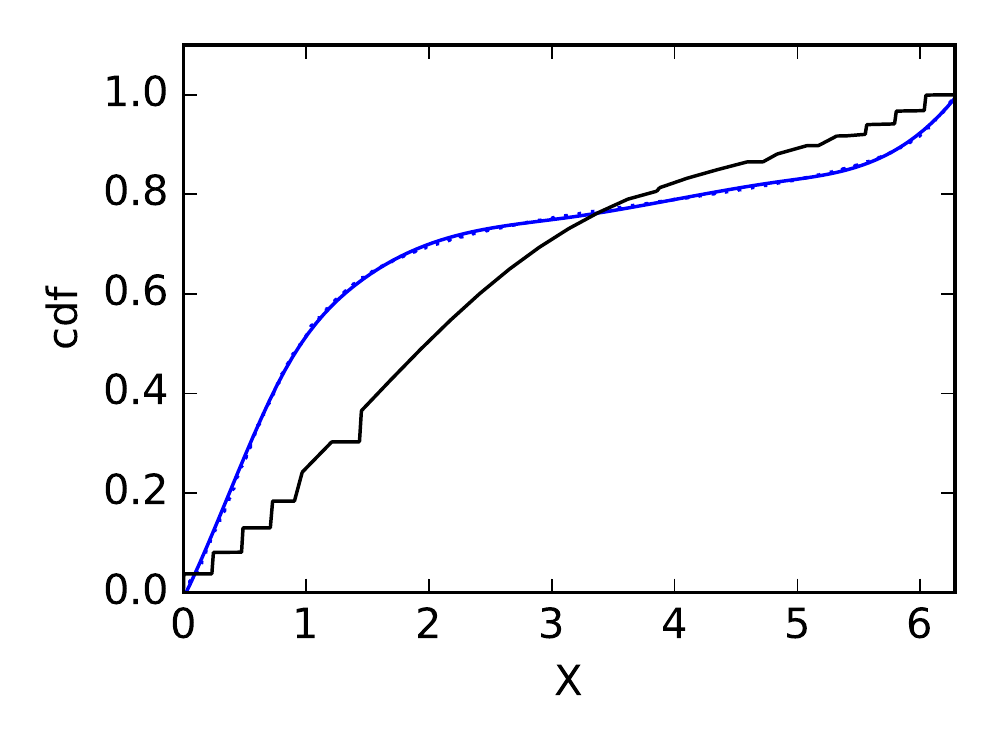}
	\includegraphics[width=0.45\textwidth]{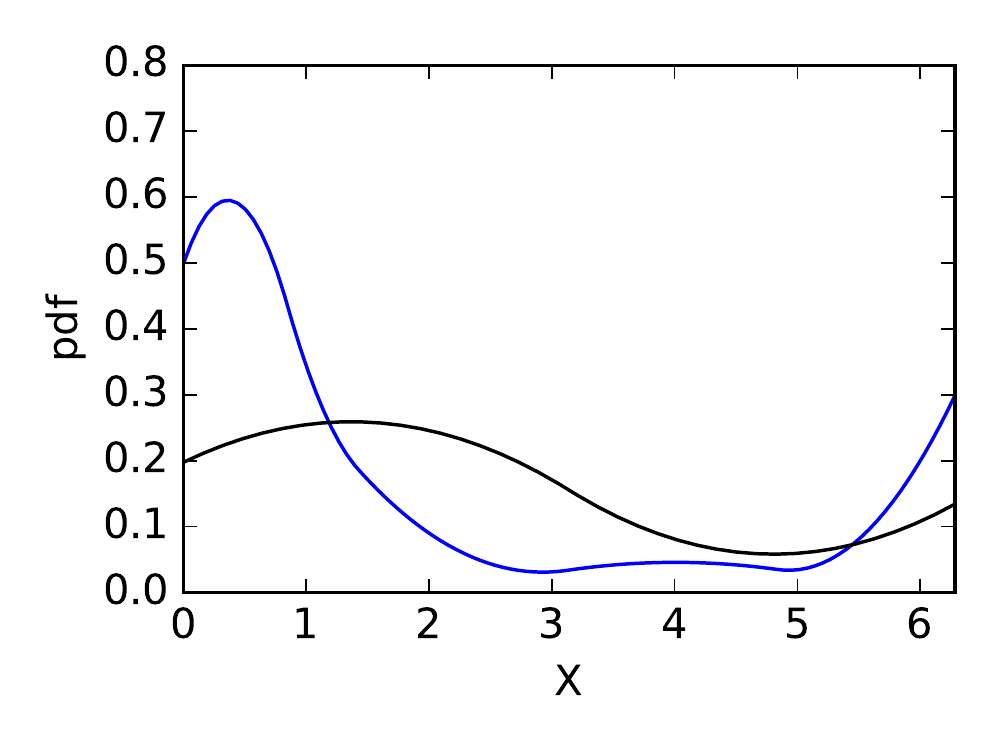}
	\caption{Plane rotator: the pdf and cdf for initial distribution  $\Nrm(\pi/4,3\pi/4)$.  The  plots show initial (black) and final (blue) distributions. The pdf is computed by differentiating a spline approximation to the cdf.} \label{cpdf}
\end{figure}
\subsection{Viscous Burgers equation}
Consider  the following mean-field SDE for a parameter $\sigma>0$:
\[
dX^\mu(t)=\int_\real \pp{1- H(X^\mu(t)-y)}\,P^\mu_t(dy) \,dt %
+ \sigma\,dW(t),
\]
where $H$ is the Heaviside step function with $H(x)=0$ for $x< 0$ and $=1$ for $x\ge 0$, and an initial distribution $X^\mu(0)$ is prescribed. The drift term here can also be written as $\bar{a}(X,t)=\prob{X^\mu(t) <X}$. Let $X^\mu(t)$ have cumulative distribution function (cdf) $u(t,x)$; then $V(t,x)=1-u(t,x)$ satisfies the viscous Burgers equation
\[
\frac{\partial V}{\partial t}%
=\frac 12 \,\sigma^2 \,\frac{\partial^2 V}{\partial x^2} - V\,\frac{\partial V}{\partial x},\qquad x\in\real.
\]
In general, the solution of the  initial-value problem for viscous Burgers equation can be written as the difference of two cdfs defined by initial-value problems for a mean-field SDE \cite{Bossy1997-fs}.

For $X^\mu(0)$ equal to delta measure at zero, the exact cdf is  $u(0,x)=H(x)$ and
\begin{equation}\label{burger_exact}
u(t,x)=\frac{\operatorname{erfc} (-x/\sqrt{2 \,\sigma^2\,t})}%
{\operatorname{erfc}(-x/\sqrt{2\,\sigma^2\,t})+\exp((t-2\,x)/2\,\sigma^2) (2-\operatorname{erfc}((t-x)/\sqrt{2\,\sigma^2\,t})},
\end{equation}
where $\operatorname{erfc}$ denotes the complementary error function \cite{Bossy1997-fs}. We see in particular the solution represents a soliton travelling to the right with speed $1/2$.

For the GQ methods, this problem presents two challenges. First, the mean-field term cannot be factored out as in \cref{factor} and $P_t^\mu (H( \cdot-X^\mu(t)))$ must be evaluated by quadrature for each particle representing $X^\mu(t)$. This increases computation time as $m$ quadratures are needed at each step, instead of one. The lack of structure also means GQ2 cannot be used.

Second, the Heaviside function has a jump discontinuity at $x=0$ and this lack of smoothness is evident in experiments.
Introduce the regularised function
\[
1-H(x)%
\approx \frac 12 \operatorname{erfc}(x/\ell),\qquad x\in\real,
\]
for a length scale $\ell>0$.
The equation
\begin{equation}\label{burger_reg}
dX^\mu(t)=\int_\real \frac 12 \operatorname{erfc}\pp{\frac{X^\mu(t)-y}{\ell}}\,P^\mu_t(dy) \,dt + \sigma\,dW(t)
\end{equation}
has smooth bounded coefficients and the behaviour of the GQ algorithms is shown in \cref{conv_burger_good}. The convergence behaviour is broadly in line with the theory for $\phi(x)=x^2$, though GQ1e looses accuracy for small $\tstep$ when $\ell$ is reduced to $\ell=0.001$ from $\ell=0.1$ and the drift more closely resembles the Heaviside function. GQ1 and GQ1e accurately compute the first moment, which gives the centre of the soliton at $x=1/2$, to high accuracy (the error is $10^{-12}$ even for $\tstep=0.05$ and $\ell=0.001$; not shown in the figures). \cref{cdf_burger} shows a comparison of the cdf of GQ1e using $\ell=0.001$  with the exact cdf for $\tstep=3\times 10^{-4}$ with $74$ quadrature points. The two agree with an $L^1(\real)$ error of approximately $10^{-2}$.

\begin{figure}
	\centering
	\includegraphics[width=0.45\textwidth]{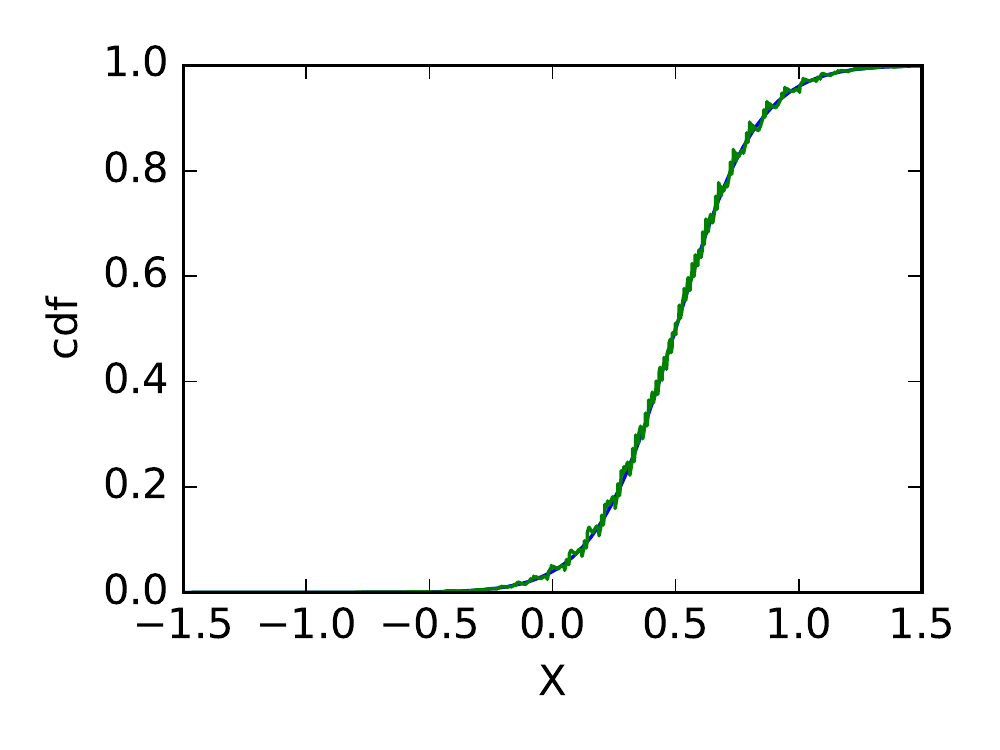}
	\includegraphics[width=0.45\textwidth]{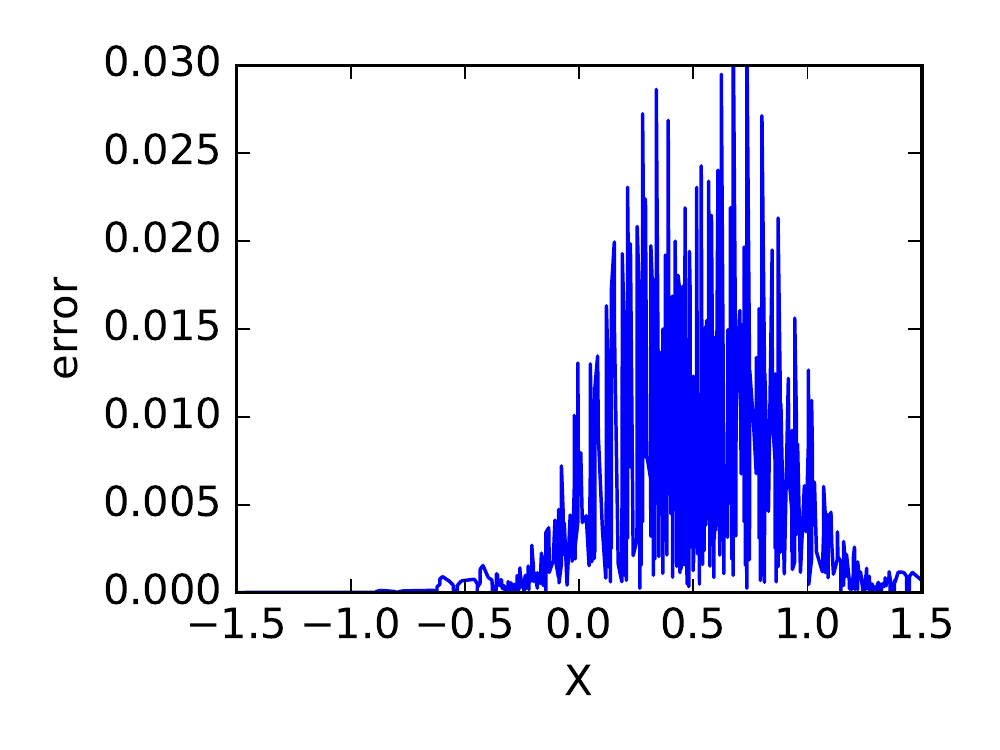}\\
	\caption{Burgers equation for $\frac 12\sigma^2=0.1$: Comparison of the exact cdf at $t=1$ given by \cref{burger_exact} and the numerical approximation by GQ1e of \cref{burger_reg} with  for $\tstep=3\times 10^{-4}$ and $\ell=10^{-3}$ (using 74 quadrature points).} \label{cdf_burger}
\end{figure}
\begin{figure}
	\centering
	\includegraphics[width=0.45\textwidth]{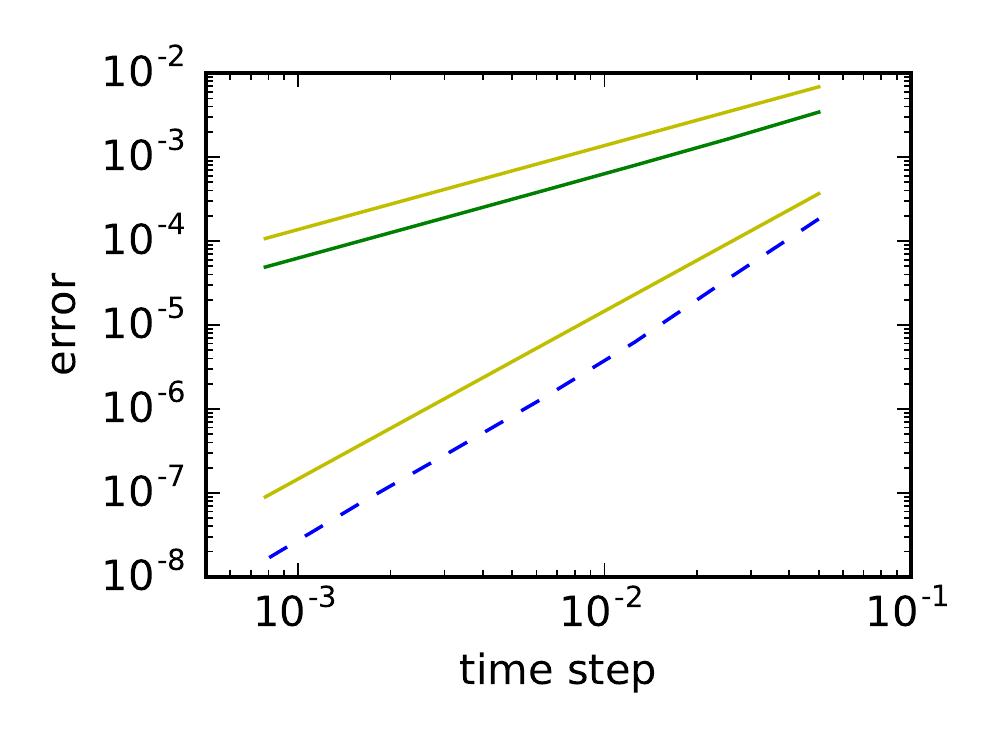}
	\includegraphics[width=0.45\textwidth]{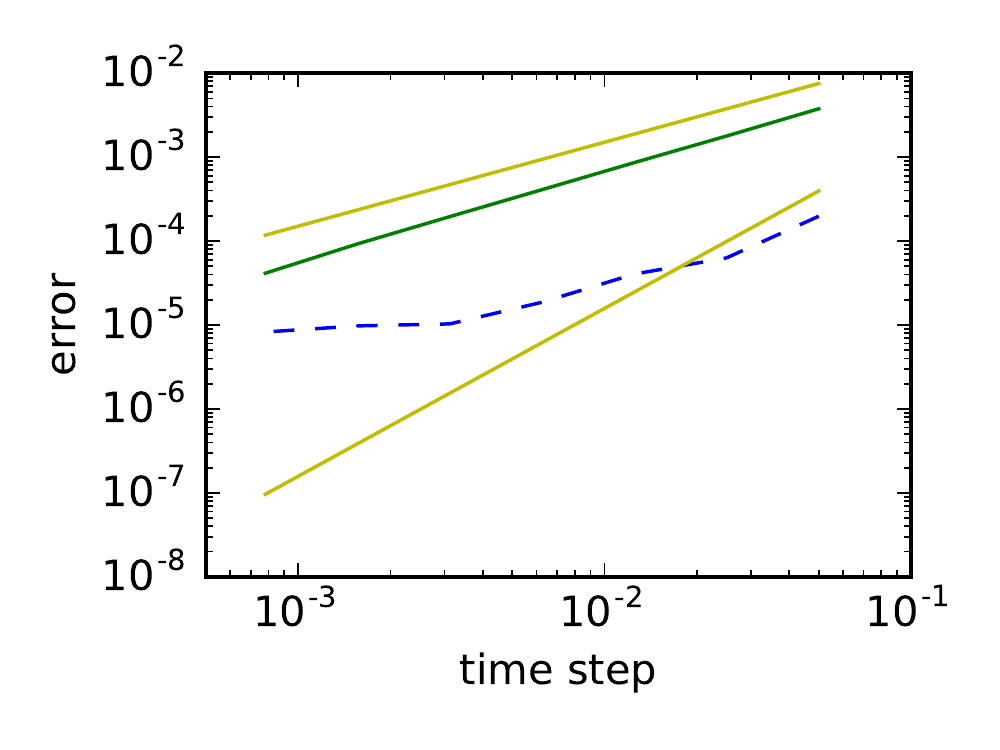}\\
	\caption{Burgers equation for $\frac 12\sigma^2=0.1$: The error in approximating the second moment of \cref{burger_reg} for $\ell=0.1$ (left) and $\ell=0.001$ (right). The green line marks GQ1 and the blue dashed-line marks GQ1e.} \label{conv_burger_good}
\end{figure}

\section{Conclusion}\label{conc}

We have derived a time-stepping method based on Gauss quadrature for approximating the probability distribution of the solution of mean-field SDEs at a fixed time. The work per time step is dominated by the eigenvalue problem for determining the Gauss quadrature. The total work required depends on the smoothness of the underlying problem and in the best case is $\order{\varepsilon^{-1/p}\,\abs{\log \varepsilon}^{3}}$ operations when the underlying time-stepping method has $p$th order accuracy.

Though very effective for one-dimensional mean-field SDEs, their dependence on Gauss quadrature means the presented methods are difficult to extend to higher dimensions. The available methods for higher dimensions include \cite{Muller-Gronbach2015-vv,Ricketson2015-xv,McMurray2015-lv} and   are not as efficient. One-dimensional mean-field SDEs remain an interesting case due to their use in understanding high-dimensional interacting particle systems and the proposed methods are far more efficient than currently available methods.

The drift $a$ and diffusion $b$ in this paper are assumed to be bounded with bounded derivatives, which is unrealistic for many problems (including those  in \cref{num} with polynomial $a$ and $b$). Much work is currently being undertaken to extend the numerical analysis of SDEs to non-Lipschitz problems (for example, \cite{Hutzenthaler2014-ma,Hutzenthaler2015-eg}). Some of this will carry over to the Gauss-quadrature methods and mean-field SDEs, though nice properties such as \cref{mf_mom} (boundedness of exponential moments for Euler--Maruyama) no longer hold in general.  Some extensions are  presented in
\cite{Muller-Gronbach2015-vv}, who also consider bounded coefficients but allow more general regularity conditions on the test functions than presented here. They also provide  a non-uniform time-stepping scheme that allows more efficient approximation of less smooth problems.

\bibliographystyle{siamplain}
\bibliography{weak_new}

\end{document}